\documentclass [12pt,oneside]{article}
\usepackage{a4,latexsym,exscale,theorem,epsfig}
\usepackage{amssymb,psfrag,epsf,amsmath,verbatim,bbm,float}
\usepackage{listings}
\usepackage{color}
\usepackage{mathbbol}
\newcounter{lemma}
\newtheorem{lemma}{Lemma}

\newtheorem{definition}[lemma]{Definition}

\newtheorem{theorem}[lemma]{Theorem}
\newtheorem{remark}[lemma]{Remark}
\newtheorem{hyp}[lemma]{Hypothesis}

\begin{document}

\newenvironment{proof}{\hspace{-\parindent}{\bf{}Proof:}}
{\hspace*{\fill}$\blacksquare$}
\newcommand {\eps} {\varepsilon}
\newcommand {\Z} {\mathbbm{Z}}
\newcommand {\R} {\mathbbm{R}}
\newcommand {\N} {\mathbbm{N}}
\newcommand {\ang} {\measuredangle}
\newcommand {\e} {{\rm{e}}}
\newcommand {\rank} {{\rm{rank}}}
\newcommand {\Span} {{\rm{span}}}
\newcommand {\card} {{\rm{card}}}
\newcommand {\OO} {\mathcal{O}}
\newcommand {\nr} {\mathcal{N}}
\newcommand {\Sn}[1] {\mathcal{S}^{#1}}
\newcommand {\range} {\mathcal{R}}
\newcommand {\kernel} {\mathcal{N}}
\newcommand{\one}{\mathbb{1}}
\renewcommand{\thefootnote}{\fnsymbol{footnote}}

\newcommand{\todo}[1]{\par\vspace{2mm} \noindent
\marginpar{todo}
\framebox{\begin{minipage}[H]{0.75\textwidth}
\tt #1 \end{minipage}}\vspace{2mm}\par}


\title{\bf Continuation and collapse of homoclinic tangles}

\author{Wolf-J\"urgen Beyn\footnotemark[1]\qquad
Thorsten H\"uls\footnotemark[1]\\\normalsize
Department of Mathematics, Bielefeld University\\\normalsize
POB 100131, 33501 Bielefeld, Germany\\\normalsize
\texttt{beyn@math.uni-bielefeld.de}\qquad
\texttt{huels@math.uni-bielefeld.de}
}
\footnotetext[1]{Supported by
CRC 701 'Spectral Structures and Topological Methods in Mathematics'.}
\maketitle


\begin{abstract}
By a classical theorem transversal homoclinic points of maps lead
to shift dynamics on a maximal invariant set, also referred to as
a homoclinic tangle. In this paper we study the fate of 
homoclinic tangles in parameterized systems from the viewpoint of
numerical continuation and bifurcation theory. The bifurcation result
shows that the maximal invariant set near a homoclinic tangency, where
two homoclinic tangles collide, can be characterized by a system
of bifurcation equations that is indexed by a symbolic sequence.
For the H\'{e}non family we investigate in detail the bifurcation 
structure of multi-humped orbits originating from several tangencies.
The homoclinic network found by numerical continuation is explained
by combining our bifurcation result with graph-theoretical arguments.

\end{abstract}

\noindent \textbf{Keywords:} Homoclinic tangency, symbolic dynamics, numerical continuation, bifurcation of homoclinic orbits. 

\noindent \textbf{AMS Subject Classification:} 37N30, 65P20, 65P30


\section{Introduction}
\label{S1}

We consider parameter dependent, discrete time dynamical systems of
the form 
\begin{equation}\label{a1}
x_{n+1} = f(x_n,\lambda),\quad n\in \Z,
\end{equation}
where $f(\cdot,\lambda),\lambda\in \R$ are smooth diffeomorphisms 
in $\R^k$. We assume that the system \eqref{a1} has a smooth
branch of hyperbolic fixed points and our main interest is in branches of homoclinic orbits that return to these fixed points. Generically
one finds turning points on these branches which correspond to homoclinic
tangencies where stable and unstable manifolds  of the fixed point
intersect nontransversally, see \cite{kl98a}, \cite{bhkz04}
for the precise relation.
While the dynamics near transversal intersections are well understood
through the celebrated Smale-Shilnikov-Birkhoff Theorem
(see \cite{sm67},\cite{sh67},\cite{gh90},\cite{p00}), the picture 
near homoclinic tangencies still seems to be far from being complete.
Among the many references, we mention the monograph \cite{pt93}
which contains a detailed geometrical study of the bifurcations that
occur near tangencies, the work \cite{gts03} which supports the
generic occurrence of homoclinic tangencies of all orders,
and the paper \cite{kn06} which proves shift dynamics in arbitrarily
small neighborhoods of the tangency.
We also mention that homoclinic orbits with a tangency can be computed
numerically in a robust way by solving boundary value problems on  a finite interval, and that the errors caused by this approximation have been completely analyzed, see \cite{kl97, kl98a, bhkz04}.

In this paper, we consider bifurcations of tangential homoclinic orbits
from a local as well as from a global viewpoint. 

The local study determines the elements of the maximal invariant
set in a neighborhood of the tangent orbit and of the  critical parameter 
from a set of bifurcation equations. Using the same shift space
as in the transversal case, we associate with any sequence of symbols a 
bifurcation equation that describes those branches of orbits
which have the return pattern of the symbolic sequence.
Such a result does not fully resolve the dynamics
near tangencies, but reduces the problem to a set of perturbed bifurcation equations for which the unperturbed form is known (similar to Liapunov-Schmidt reduction). For example, multi-humped homoclinic orbits that enter
and leave a neighborhood of the fixed point several times, relate to
a perturbed system of hilltop bifurcations, cf.\ \cite{gs85}.
Our main results will be stated in Section \ref{S2} with the
proofs deferred to Sections \ref{S5}, \ref{S6}.

The global viewpoint asks for possible bifurcations of
multi-humped orbits that are known to emerge from the tangencies. 
We take the H\'{e}non family as a model equation for a detailed numerical
study of the homoclinic network that arises from a total of $4$
primary homoclinic tangencies. It turns out that the connected components
of this network are by no means arbitrary. Rather, they
follow certain rules governing the bifurcations of multi-humped orbits. 
Combining these rules with graph-theoretical and combinatorial arguments
allows to predict the structure to a large extent. Only some fine 
details are left to numerical computations as will be demonstrated in
Sections \ref{S3} and \ref{S4}. 


\section{Setting of the problem and main results}
\label{S2}
The aim of this section is to state our main result on bifurcation
equations near homoclinic tangencies.
We first describe the setting and state our assumptions:
\begin{itemize}
\item [\textbf{A1}] $f\in\mathcal{C}^\infty(\R^k\times \Lambda_0, \R^k)$
for some open set $\Lambda_0 \subset \R$  and 
$f(\cdot,\lambda)$ is a diffeomorphism for all $\lambda \in \Lambda_0$,
\item [\textbf{A2}] $f(\xi(\lambda),\lambda) = \xi(\lambda)$ 
for some smooth branch $\xi(\lambda) \in \R^k, \lambda\in \Lambda_0$,
\item [\textbf{A3}] $f_x(\xi(\lambda),\lambda) \in \R^{k,k}$ is
  hyperbolic for all 
  $\lambda \in \Lambda_0$.
\end{itemize}
Clearly, if $\xi_0$ is a hyperbolic fixed point of 
$f(\cdot,\lambda_0)$ for some $\lambda_0\in \R$ then \textbf{A2} and 
\textbf{A3} follow for some neighborhood $\Lambda_0$ of  $\lambda_0$.
Replacing $f$ by $g(x,\lambda) =
f(x+\xi(\lambda+\lambda_0),\lambda+\lambda_0) - \xi(\lambda+\lambda_0)$
shows that \textbf{A2}, \textbf{A3} can be assumed to hold for the trivial 
branch $\xi(\lambda)=0$ and for a neighborhood $\Lambda_0$ of zero. This will
be our standing assumption throughout Sections \ref{S5} and \ref{S6}. 

It is well known that transversal homoclinic orbits lead to chaotic dynamics
on a nearby invariant set commonly referred to as a homoclinic tangle. 
Let us first assume that this situation occurs at some parameter value
 $\tilde \lambda\in \Lambda_0$. 
\begin{itemize}
\item [\textbf{A4}] For some $\tilde \lambda \in \Lambda_0$ there exists a
  nontrivial homoclinic orbit $\tilde x_\Z =(\tilde x_n)_{n\in \Z}$,
  i.e.\ $\lim_{n \to \pm \infty} \tilde x_n = \tilde \xi:=\xi(\tilde \lambda)$
and $\tilde x_n \neq \tilde \xi$ for some $n \in \Z$.
  This orbit is \textit{transversal} 
  in the sense that the variational equation
  \begin{equation}\label{var-eq1}
  y_{n+1} = f_x(\tilde x_n,\tilde \lambda) y_n,\quad n \in \Z
  \end{equation}
  has no nontrivial bounded solution on $\Z$.
\end{itemize}

In this case the stable and the unstable manifold of $\tilde \xi$ intersect
transversally at each $\tilde{x}_n$ and the set
$
\tilde H = \{\tilde x_n\}_{n\in\Z} \cup \{\tilde \xi\}
$
is hyperbolic, cf.\ \cite{pa88}. Moreover, there exists 
an open neighborhood $U$ of $\tilde H$ such that the dynamics on the
maximal invariant set 
\begin{equation} \label{maxinv}
M(U,\tilde \lambda) = \{x\in U: f^n(x,\tilde \lambda) \in U\ \forall n
\in \Z\}
\end{equation}
is conjugate to a subshift of finite type (see 
 the Smale-Shilnikov-Birkhoff Homoclinic Theorem in
\cite{gh90}  and \cite[Chapter 5]{p00} for a proof). 
To be precise, let $N\ge 2$ and let 
$$
S_N = \{0,1,\dots,N-1\}^\Z
$$
be the shift space with $N$ symbols which is compact w.r.t.\ the metric
\begin{equation} \label{defmetric}
d(s,t) = \sum_{j \in \Z} 2^{-|j|} |s_j - t_j|,\quad s = (s_j)_{j\in\Z}
\in S_N.
\end{equation}
Let $\beta$ be the Bernoulli shift
$$
\beta(s)_i = s_{i+1},\quad i \in \Z,\quad s \in S_N.
$$
Consider a special subshift of finite type, see \cite{lm95} 
$$
\Omega_N = \{s\in S_N : A^{(N)}_{s_i,s_{i+1}} = 1 \ \forall i \in \Z\}
$$
generated by the $N\times N$ binary matrix
\small $$
A^{(N)} = 
\begin{pmatrix}
1 & 1 & 0 & \cdots & 0\\
0 & 0 & 1 & \ddots & \vdots\\  
\vdots && \ddots & 1 & 0\\
0 &&&\ddots&1\\
1 & 0 &\cdots& \cdots& 0
\end{pmatrix}
\in \{0,1\}^{N\times N}.
$$ \normalsize
Then there exists a neighborhood $U$ of $\tilde H$, an integer $N \ge 2$ and a
homeomorphism $h:\Omega_N \to M(U,\tilde \lambda)$ such that 
\begin{equation}\label{1.12}
f(\cdot,\tilde \lambda) \circ h = h \circ \beta \quad \text{in } \Omega_N.
\end{equation}

A continuation of the transversal homoclinic orbit w.r.t.\ the
parameter $\lambda$ leads to a curve of homoclinic orbits
that typically exhibits turning points. As an example we refer
to Figure \ref{fig1} for the H\'{e}non map. Parts of the branch that
can be parametrized by $\lambda$ belong to transversal homoclinic
orbits while (quadratic) turning points correspond to homoclinic
orbits with a (quadratic) tangency, see Theorem \ref{Th1.2} for
a precise statement. In this case we replace Assumption \textbf{A4}
by
\begin{itemize}
\item [\textbf{B4}] For some $\bar \lambda \in \Lambda_0$ there exists a
  nontrivial homoclinic orbit $\bar x_\Z =(\bar x_n)_{n\in \Z}$ converging
towards $\bar{\xi}=\xi(\bar{\lambda})$.
    The orbit is \textit{tangential} 
  in the sense that the variational equation
  \begin{equation} \label{vareq}
  y_{n+1} = f_x(\bar x_n,\bar \lambda) y_n,\quad n \in \Z
  \end{equation}
  has a non-trivial solution 
 $u_\Z=(u_n)_{n\in \Z}$ that is unique up to constant multiples.
\end{itemize}
Since the fixed point stays hyperbolic we have exponential decay for both
the orbit and the solution of \eqref{vareq}, i.e.\ for some
$\alpha,C_e >0$
\begin{equation} \label{exdec}
\|\bar{x}_n - \bar{\xi}\| +\|u_n\| \le C_e \e^{-\alpha |n|}, \quad n\in \Z.
\end{equation}
Therefore,  we may normalize 
\begin{equation} \label{unormalize}
\|u_\Z\|^2_{\ell^2} = \langle u_\Z, u_\Z\rangle^2_{\ell^2} 
= \sum_{n\in\Z} u_n^T u_n = 1.
\end{equation}
In the following we use $\langle\cdot,\cdot \rangle$ to denote the inner
product in $\ell^2$.
The assumption on \eqref{vareq} in  \textbf{B4} holds if and 
only if the tangent spaces of the stable and unstable manifold have
a one-dimensional intersection, i.e.\
$$
T_{\bar x_n} W^s(\bar \xi) \cap T_{\bar x_n} W^u(\bar \xi) =
\Span(\bar u_n),\quad n \in \Z.
$$
We refer to Theorem \ref{Th1.2} and to \cite[Appendix]{kl98a} for a more general statement.

Consider open neighborhoods  $U\subset \R^k$  of 
$H = \{\bar x_n\}_{n\in\Z} \cup \{\bar \xi\}$ and $\Lambda \subset \Lambda_0$
of $\bar \lambda$, respectively.
Our main interest is in  the dynamics on the
 maximal invariant set
$$
 M(U,\Lambda) = \{(x,\lambda)\in U\times \Lambda: f^n(x,\lambda)
\in U\ \forall n \in \Z\}.
$$ 
As in the transversal case, we will still work with the subshift
$(\Omega_N, \beta)$ but the conjugacy \eqref{1.12} will be replaced by
a set of bifurcation equations. 
For any $s \in \Omega_N$ define the index set
\begin{equation} \label{indexset}
I(s) = \{n\in\Z: s_n = 1\},
\end{equation}
and note that $I:\Omega_N\to \Z(N) \subset 2^\Z$ is bijective, where
\begin{equation} \label{zndef}
\Z(N) = \{J \subset \Z: |j-k| \ge N\ \forall j,k\in J, j\neq k\}.
\end{equation}
With any $s\in \Omega_N$ we associate the Banach space
$$
\ell^\infty(s) = \{\tau \in \R^{I(s)}: \|\tau\|_\infty < \infty\},
$$
$$
\|\tau\|_\infty = \sup_{\ell \in I(s)}|\tau(\ell)|,\quad
B_\rho = \{\tau \in \ell^\infty(s): \|\tau\|_\infty \le \rho\}.
$$

Our aim is to determine the elements of $M(U,\Lambda)$ from a set of
bifurcation equations
\begin{equation}\label{1.15}
g_s(\tau,\lambda) = 0, \quad \tau\in B_{\rho_{\tau}},\lambda\in \Lambda,
\end{equation}
where $s \in \Omega_N$, $\rho_\tau > 0$ is independent of $s$ and 
$$
g_s:
\begin{array}{rcl}
B_{\rho_\tau}\times \Lambda & \to & \ell^\infty(s)\\
(\tau,\lambda) & \mapsto & g_s(\tau,\lambda)  
\end{array}
$$
is a sufficiently smooth map. Note that \eqref{1.15} constitutes a finite or
 an infinite
system of equations depending on the cardinality of  $I(s)$.

In order to formulate the precise statement we define the pseudo
orbits 
\begin{equation}\label{1.16}
p_n(s) = \bar{\xi}+ \sum_{\ell  \in I(s)} (\bar
x_{n-\ell}-\bar{\xi}),\quad n \in \Z.
\end{equation}
Equation \eqref{exdec} shows that $p_\Z(s)$ is a bounded sequence, in particular
\begin{equation} \label{pseudoest}
\|p_n(s)-\bar{\xi}\| \le \bar{C}= C_e \frac{1+\e^{-\alpha}}{1-\e^{-\alpha}}, \quad n\in \Z.
\end{equation} 
Setting $\bar \xi_n = \bar \xi$ for all $n\in \Z$ we write 
\eqref{1.16} more formally as
$$
p_\Z(s) = \bar{\xi}_\Z + \sum_{\ell \in I(s)} \beta^{-\ell}(\bar
x_\Z-\bar{\xi}_\Z).
$$
Here and it what follows we use the symbol $\beta$  to denote the shift
of sequences in $\R^k$. Thus $\beta$ acts as an operator in sequence spaces
such as  $\ell^{p}(\R^k), 1\le p \le \infty$. 

Similarly, for every $\tau\in
\ell^\infty(s)$ we define the bounded sequence 
\begin{equation}\label{1.17}
v_\Z(s,\tau) = \sum_{\ell \in I(s)} \tau_\ell \beta^{-\ell}u_\Z .
\end{equation}
Note that the sequence $p_\Z(s)$  has humps at the positions
defined by $I(s)$ and that $p_\Z(s)$ is a pseudo orbit of
$f(\cdot,\bar \lambda)$ with a small error, see Lemma \ref{L5.3}.
 The term $v_\Z(s,\tau)$
shifts the solution of the variational equation to the positions defined
by $I(s)$ and combines them linearly.

\begin{theorem}\label{Th1.1}
Let assumptions {\bf A1 - A3} and {\bf B4} hold.
Then there exist constants $0<r_{\tau}\le\rho_\tau$, $N \in \N$ and 
neighborhoods $U$
of $H$, $\Lambda$ of $\bar \lambda$ and for any $s \in \Omega_N$ smooth
functions 
\begin{eqnarray*}
g_s &:& B_{\rho_\tau} \times \Lambda \to \ell^\infty(s),\\
x_{\Z,s} &:& B_{\rho_\tau} \times \Lambda \to \ell^\infty(\R^k)  
\end{eqnarray*}
with the following properties.
\begin{itemize}
\item[(i)] For any point $(y_0,\lambda) \in
M(U,\Lambda)$ with orbit $y_{n} = f^n(y_0,\lambda)$, $n \in \Z$ there
exists an index  $\nu \in \Z$ and
 elements $s \in \Omega_N$, $\tau \in B_{\rho_\tau} \subset
\ell^\infty(s)$  such that 
\begin{align}  \ \label{1.20}
\beta^\nu y_\Z &= x_{\Z,s}(\tau,\lambda) + p_\Z(s) + v_\Z(s,\tau),\\
g_s(\tau,\lambda) &= 0. \label{1.21}
\end{align}  

\item[(ii)]
Conversely, if $s \in \Omega_N$, $\tau \in B_{r_\tau}\subset \ell^{\infty}(s)$, $\lambda
\in \Lambda$ satisfy \eqref{1.21}, 
then there exists $\nu \in \Z$ such that $(y_n,\lambda)_{n\in \Z}$, 
with $y_\Z$ given by \eqref{1.20}, belongs to $M(U,\Lambda)$.
\end{itemize}
\end{theorem}
\begin{remark}
Theorem \ref{Th1.1} reduces the study of $M(U,\Lambda)$ to the set of
bifurcation equations \eqref{1.21} with a symbolic index $s \in \Omega_N$. 
It may be regarded as a type of Liapunov-Schmidt reduction though we have
not formally put it into this framework. 
The construction of the neighborhood $U\times \Lambda$ uses some features 
from the transversal case \cite[Theorem 5.1]{p00}, but is
considerably more involved, see Sections \ref{S5} and \ref{S6}.
We also note
that we were not able to prove that one can take 
$r_{\tau}=\rho_{\tau}$ which would give a complete characterization of
 $M(U,\Lambda)$ in terms of \eqref{1.20},\eqref{1.21}.
Another issue which has not yet been resolved, is continuous dependence of the 
functions $x_{\Z,s}$ and $g_s$ on the symbolic sequence $s$ with respect to the metric \eqref{defmetric}.
 
\end{remark}
The functions $g_s$ and $x_{\Z,s}$ have several properties that we
discuss next. \\
Due to \textbf{B4} the adjoint equation 
\begin{equation}\label{1.22}
y_{n+1}^T f_x(\bar x_{n+1},\bar \lambda) = y_n^T,\quad n \in \Z
\end{equation}
has a non-trivial solution $w_\Z$ that is unique up to constant multiples,
cf.\ \cite[Section 2]{pa88}.
It decays exponentially as in \eqref{exdec} and can thus be normalized such
that $\|w_\Z\|_{\ell^2} = 1$. Without loss of generality we take $C_e$ in 
\eqref{exdec} such that 
\begin{equation} \label{exdectrans}
\|w_n \| \le C_e \e^{-\alpha |n|}, \quad n \in \Z.
\end{equation}

As is shown in \cite{kl98a} the quantities
\begin{equation}\label{cdef}
c_\lambda = \langle w_\Z, (f_\lambda(\bar x_n,\bar \lambda))_{n\in\Z}
\rangle, \quad  
c_x = \frac{1}{2} \langle w_\Z,(f_{xx}(\bar x_n,\bar \lambda)
u_n^2)_{n\in\Z}\rangle
\end{equation}
characterize the behavior of the branch of homoclinic orbits which
passes through $(\bar x_\Z,\bar \lambda)$.

\begin{theorem}\label{Th1.2}
The operator $F:\ell^\infty(\R^k)\times \R \to \ell^\infty(\R^k)$
defined by
\begin{equation} \label{Fdef}
F(x_\Z,\lambda) = \big( (x_{n+1} - f(x_n,\lambda))_{n\in\Z}\big) 
\end{equation}
has a limit point at $(\bar x_\Z,\bar \lambda)$ in the sense that 
$F(\bar x_\Z,\bar \lambda) = 0$ and 
$$
\nr(D_x F(\bar x_\Z,\bar \lambda)) = \Span\{u_\Z\}.
$$
The limit point is transversal, i.e.\
$$
D_\lambda F(\bar x_\Z,\bar \lambda) \notin \range(D_x F(\bar x_\Z,\bar
\lambda)) \quad \text{if and only if}\quad c_\lambda \neq 0. 
$$ 
Moreover, it is a quadratic turning point, i.e.\
$$
D^2_x F(\bar x_\Z,\bar \lambda) u_\Z^2 \notin \range(D_x F(\bar
x_\Z,\bar \lambda)) \quad \text{if and only if} \quad c_x \neq 0.
$$
\end{theorem}

\begin{remark}
For transversal homoclinic orbits ($\lambda \neq \bar \lambda$), the
Sacker-Sell spectrum, cf.\ \cite{ss78}, of the variational equation
\eqref{var-eq1} is a pure point spectrum 
$$
\Sigma_{\mathrm{ED}}(\lambda) = \{|\mu| : \mu \in
\sigma(f_x(\xi(\lambda),\lambda))\}.
$$
At a turning point, we find a spectral explosion to a continuous
Sacker-Sell spectrum 
$$
\Sigma_{\mathrm{ED}}(\bar \lambda) = \{|\mu| : \mu \in
\sigma(f_x(\bar \xi,\bar \lambda))\} \cup [\mu_s,\mu_u],
$$
where
$$
\mu_s = \max\{|\mu| < 1: \mu \in \sigma(f_x(\bar \xi,\bar
\lambda))\},\quad 
\mu_u = \min\{|\mu| > 1: \mu \in \sigma(f_x(\bar \xi,\bar
\lambda))\}.
$$
\end{remark}

Our second result shows that the constants $c_\lambda$, $c_x$ 
play an important role in the behavior of the bifurcation function
$g_s(\tau, \lambda)$.

\begin{theorem}\label{Th1.3}
Let the assumptions of Theorem \ref{Th1.1} hold. Then the functions
$x_{\Z,s}$, $g_s$ have the following properties 
\begin{itemize}
\item [(i)]
  \begin{align} \label{xshift}
   x_{\Z,\beta s} (\beta \tau,\lambda) &= \beta
   x_{\Z,s}(\tau,\lambda), \quad \text{where} \; I(\beta s)=I(s)-1,
\\ \label{gshift}
   g_{\beta s} (\beta \tau,\lambda) &= \beta g_s(\tau,\lambda). 
  \end{align}
\item [(ii)] For some constants $C>0$, $\alpha>0$, independent of
  $s,N$ and $\ell \in I(s)$
\begin{equation} \label{bifest}
\big| g_s(\tau,\lambda)_{\ell} - (c_\lambda(\lambda-\bar \lambda) + c_x
\tau^2_\ell)\big| \le C\big((\lambda-\bar{\lambda})^2 + 
|\lambda-\bar{\lambda}| \|\tau\|_\infty + \|\tau\|^3_\infty +
\e^{-\alpha N/2}\big).
\end{equation}
\end{itemize}
\end{theorem}
\begin{remark} In order to apply unfolding theory to the bifurcation
equations \eqref{1.21} one also needs estimates of
derivatives in \eqref{bifest}. Though the functions $g_{s}(\cdot,\cdot)$
come out smoothly from Theorem \ref{Th5.1}, estimating their derivatives
seems to be quite involved (cf.\ the proof of \eqref{bifest} in Section \ref{S6}) and has not yet been done.
\end{remark}  
If we consider a homoclinic symbol $s \in \Omega_N$ with $K =
\card(I(s)) < \infty$ humps, then the theorem shows that the
bifurcation equations are small perturbations of a set of $K$
identical turning point equations
$$
0 = c_\lambda (\lambda - \bar \lambda) + c_x \tau_\ell^2,\quad \ell
\in I(s).
$$
If $c_{\lambda},c_x\neq 0$ one can shift $\bar \lambda$ to zero and scale $\lambda$ and
$\tau_\ell$ such that one obtains a set of hilltop bifurcations of
order $K$, cf.\ \cite[Ch. IX, \S3]{gs85},
\begin{equation} \label{hilltop2}
0 = \lambda - \tau^2_\ell,\quad \ell \in I(s).
\end{equation}
For two-humped orbits the set $I(s)$ contains two elements and the solution
curves of \eqref{hilltop2} are shown in Figure \ref{fig5}. 
This case will be crucial for understanding the global behavior of 
homoclinic curves in the next sections.


\section{Homoclinic orbits and their continuation}
\label{S3}
A typical example, which plays the role of a normal form for quadratic
two-dimensional mappings, is the famous H\'enon map, cf.\ 
\cite{he76, mi87, de89, hk91} which is defined as
$$
f(x,\lambda) = 
\begin{pmatrix}
1+x_2 - \lambda x_1^2\\
1.4 x_1  
\end{pmatrix}.
$$
This map has fixed points
$$
\xi_\pm(\lambda) = 
\begin{pmatrix}
\nu(\lambda)\\1.4 \nu(\lambda)  
\end{pmatrix}
,\quad\text{where }
\nu(\lambda) = \frac 1{5\lambda} \left( 1 \pm \sqrt{1 +
    25\lambda}\right), 
$$
and for $\tilde \lambda = 0.35$ a transversal 
homoclinic orbit $x_\Z(\tilde \lambda)$ 
w.r.t.\ the fixed point $\xi_+(\tilde \lambda)$ exists, satisfying
Assumption \textbf{A4}. 

For numerical computations, we approximate an infinite homoclinic orbit 
$x_\Z(\tilde \lambda)$  
by a finite orbit segment $x_J$, where $J=[n_-,n_+]\cap \Z$.
The segment is determined as 
a zero of the boundary value operator 
$$
\Gamma_J(x_J,\tilde \lambda) = 
\begin{pmatrix}
x_{n+1} - f(x_n,\tilde \lambda),& \quad n=n_-,\dots,n_+-1\\
b(x_{n_-},x_{n_+})  
\end{pmatrix}.
$$
Here $b:\R^{2k}\to \R^k$ defines a boundary condition, for example 
$$
b_{\text{per}}(x,y) = x-y,\quad\text{or}\quad 
b_{\text{proj}}(x,y) =
\begin{pmatrix}
B_s (x - \bar \xi)\\
B_u (x-\bar \xi)   
\end{pmatrix},
$$
in case of periodic and projection boundary conditions, where $B_s$
and $B_u$ yield linear approximations of the stable and the
unstable manifold. Due to our hyperbolicity assumption,
$\Gamma_J(\cdot,\tilde \lambda)$ has for $J$ sufficiently large 
a unique zero in a neighborhood of the exact solution. Moreover 
approximation errors decay at an exponential rate
that depends on the type of boundary condition, cf.\ \cite{bk97b}. 

For H\'enon's map, we solve the corresponding boundary value problem,
obtain in this way an approximation of $x_\Z(\tilde \lambda)$ and continue this
orbit w.r.t.\ the parameter $\lambda$, using the method of pseudo arclength
continuation, cf.\ \cite{k77, ag90, go00b}.
In Figure \ref{fig1}, we plot the amplitude of these orbits 
$\text{amp}(x_J(\lambda)) := \left(\sum_{n\in J}\|x_n(\lambda) -
  \xi_+(\lambda)\|^2\right)^{\frac 12}$
versus the parameter. 

\begin{figure}[hbt]
  \begin{center}
   \psfrag{l0}{$0$}
   \psfrag{l1}{$1$}
   \psfrag{l2}{$2$}
   \psfrag{l3}{$3$}
   \psfrag{l30}{$\ell_{3,0} = (x_\Z^{3,0},\lambda^{3,0})$}
   \psfrag{l12}{$\ell_{1,2} = (x_\Z^{1,2},\lambda^{1,2})$}
   \psfrag{r01}{$r_{0,1} = (x_\Z^{0,1},\lambda^{0,1})$}
   \psfrag{r23}{$r_{2,3} = (x_\Z^{2,3},\lambda^{2,3})$}
   \psfrag{la}{$\lambda$}
   \psfrag{amp}{$\text{amp}$} 
   \epsfig{width=12cm, file=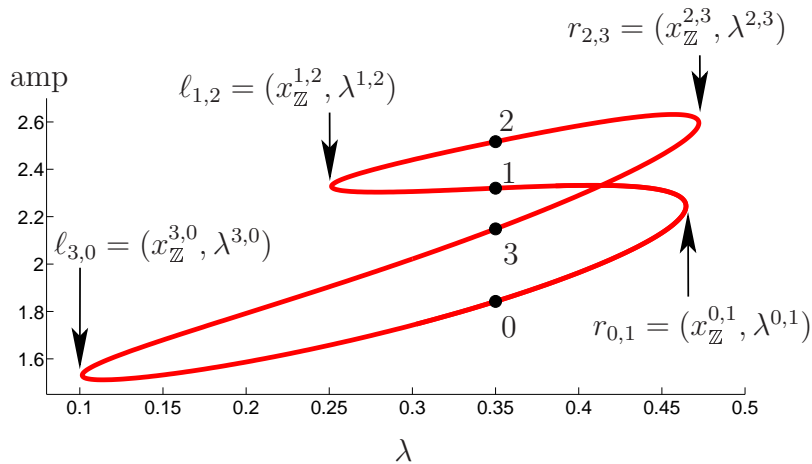}
  \end{center}
  \caption{\textit{Continuation of homoclinic H\'enon orbits. 
    At the parameter $\tilde \lambda = 0.35$ four distinct orbits
    $x_\Z^{i}$, $i\in\{0,\dots,3\}$ exist that turn into each
    other via left or right turning points. \label{fig1}}}
\end{figure}

At the value $\tilde \lambda = 0.35$ four distinct homoclinic orbits
occur that we denote by $x_\Z^{i}$,  $i\in\{0,\dots,3\}$. We choose their
index by following the order given by the continuation routine. The
orbit $x_\Z^0$ is shown in Figure \ref{fig2} together with parts of
the stable and the unstable manifold of the fixed point $\xi_+(\tilde
\lambda)$. The enlargement in this figure shows where the four orbits
lie in the intersection of manifolds.

\begin{figure}[hbt]
  \begin{center}
   \psfrag{l0}{$0$}
   \psfrag{l1}{$1$}
   \psfrag{l2}{$2$}
   \psfrag{l3}{$3$}
   \psfrag{xi}{$\xi_+$}
   \psfrag{x0}{$x_\Z^0$}
  \epsfig{width=14cm, file=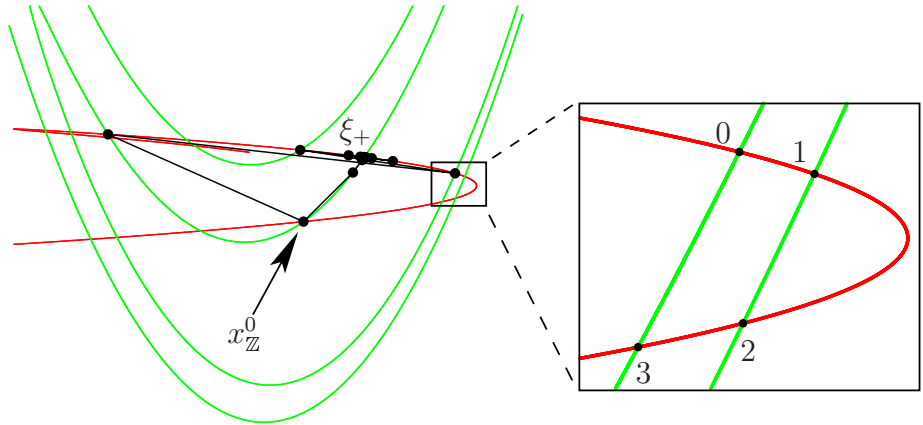}
  \end{center}
  \caption{\textit{Primary homoclinic orbit $x_\Z^0$ w.r.t.\ the fixed
      point $\xi_+$, 
    and parts of the stable manifold (green) and  
    the unstable manifold (red). The enlargement shows the
    intersections of manifolds that lead to the four homoclinic orbits
    in Figure  \ref{fig1}. \label{fig2}}}
\end{figure}

At each turning point, two orbits collide;  
with $r$ and $\ell$, we distinguish right and left turning points.
Figure \ref{fig3}
illustrates intersections of stable and unstable manifolds at these
four turning points.

\begin{figure}[hbt]
  \begin{center}
   \psfrag{l30}{$\ell_{3,0}$}
   \psfrag{l12}{$\ell_{1,2}$}
   \psfrag{r01}{$r_{0,1}$}
   \psfrag{r23}{$r_{2,3}$} 
   \psfrag{0}{$0$}
   \psfrag{1}{$1$}
   \psfrag{2}{$2$}
   \psfrag{3}{$3$}
    \psfrag{0,1}{$0,1$}
   \psfrag{1,2}{$1,2$}
   \psfrag{2,3}{$2,3$}
   \psfrag{3,0}{$3,0$}
  \epsfig{width=14cm, file=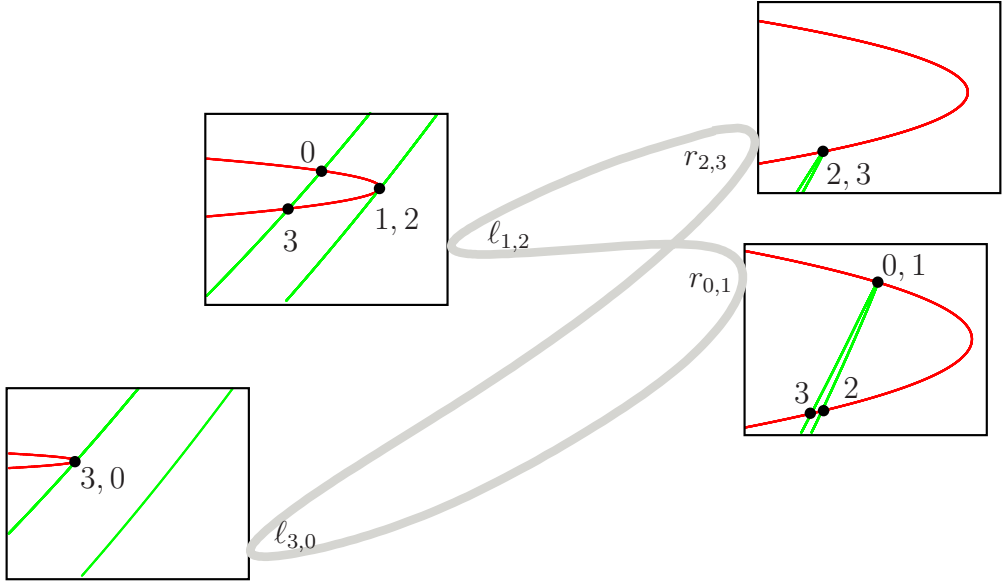}
  \end{center}
  \caption{\textit{Intersections of stable and unstable manifolds 
    at the four turning points in the
    cutout region from Figure \ref{fig2}.\label{fig3}}}
\end{figure}

Errors of turning point calculations for finite approximations of
homoclinic orbits decay exponentially fast w.r.t.\ the length of the
computed orbit segment, cf.\ \cite[Theorem 5.1.1]{kl97}.

\section{Connected components of multi-humped orbits}
\label{S4}
For H\'enon's map, we find 
four distinct transversal homoclinic orbits $x_\Z^s$, $s\in
\{0,\dots,3\}$ at $\tilde \lambda = 0.35$ and we identify $x_\Z(s)$
with its symbol $s$. Note that the orbits 
$0,1,2,3,0$ pass into each other via left (L) and right (R) turning points 
$r_{0,1}$, $\ell_{1,2}$, $r_{2,3}$, $\ell_{3,0}$, see Figure \ref{fig3}. 
The graph in Figure \ref{fig10} gives an alternative illustration 
of these transitions.

\begin{figure}[H]
  \begin{center}
    \psfrag{0}{$0$}
    \psfrag{1}{$1$}
    \psfrag{2}{$2$}
    \psfrag{3}{$3$}
    \psfrag{R}{$R$}
    \psfrag{L}{$L$}
    \epsfig{width=6cm, file=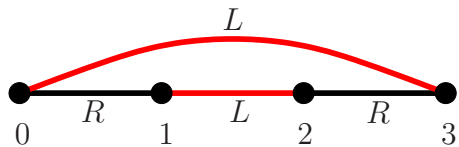}
  \end{center}
  \caption{\textit{Transition graph for one-humped orbits}.\label{fig10}}
\end{figure}

For the
construction of an $n$-humped orbit, we choose a sufficiently long
interval $J=[n_-,n_+]$ around zero and a sequence $s\in \Sn{n} :=
\{0,\dots,3\}^n$. 
We define the pseudo orbit
\begin{equation}\label{pseudo2}
\tilde x_\Z[s] := x_{(-\infty,n_+]}^{s_1} x_J^{s_2} \dots x_J^{s_{n-1}}
x_{[n_-,\infty)}^{s_n},
\end{equation}
see Figure \ref{fig4}.
Since the collection of single orbits $x_\Z^r$, $r\in\{0,\dots,3\}$
forms a hyperbolic set, the 
Shadowing-Lemma, cf.\ \cite{py99, p00} shows that the 
pseudo orbit $\tilde x_\Z(s)$ lies close to a true $n$-humped
$f$-orbit which we denote by $x_\Z(s)$.  
In $\Sn{n}$ there are $4^n$ different symbols and thus 
we expect to find $4^n$ different $n$-humped orbits $x_\Z(s)$.  
We identify these orbits with their symbol.

Note that the construction of pseudo orbits in \eqref{pseudo2}
slightly differs from \eqref{1.16}, where we add up shifted orbits. 
With both approaches, we expect to find the same shadowing orbit for
sufficiently large intervals $J$.

\begin{figure}[H]
  \begin{center}
  \psfrag{x0}{$x_\Z^0$}
   \psfrag{x1}{$x_\Z^1$}
   \psfrag{x2}{$x_\Z^2$}
   \psfrag{x3}{$x_\Z^3$}
   \psfrag{x01}{$x_\Z[01]$}
   \psfrag{x123}{$x_\Z[123]$}
  \epsfig{width=14cm, file=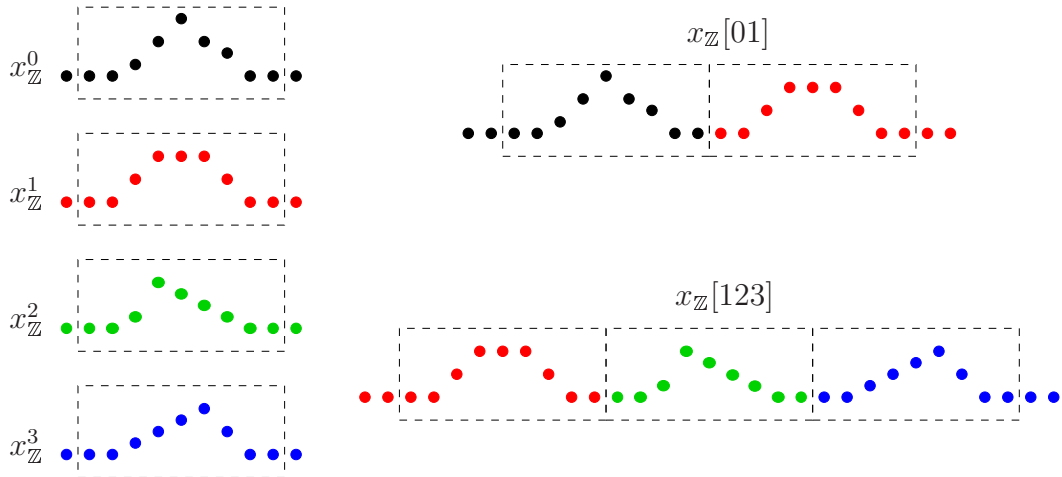}
  \end{center}
  \caption{\textit{Construction of multi-humped orbits.\label{fig4}}}
\end{figure}

Given two symbols $s,\bar s \in \Sn{n}$,
we analyze whether the $n$-humped
orbits $x_\Z(s)$ and $x_\Z(\bar s)$ can turn into each other via
continuation. 

Let us first look at the two-humped case.

\subsection{Bifurcation of two-humped orbits}\label{S41}
The continuation of two-humped orbits exhibits three closed curves of
homoclinic orbits, cf.\ Figure \ref{fig6a}, and at the parameter value
$\tilde \lambda$, $16$ different homoclinic orbits $x_\Z(s)$, $s\in
\Sn{2}$ exist.

\begin{figure}[hbt]
  \begin{center}
    \psfrag{l}{$\lambda$}
    \psfrag{amp}{amp}
    \epsfig{width=14cm, file=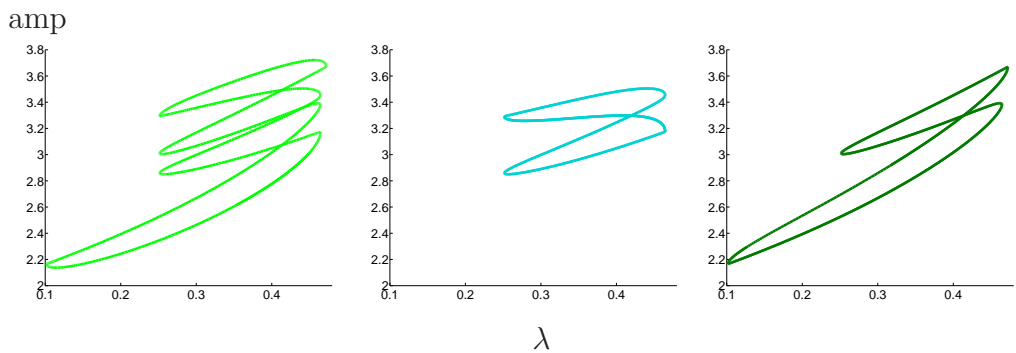}
  \end{center}
  \caption{\textit{Continuation of two-humped orbits of length $n_- =
      -20$, $n_+ = 21$. \label{fig6a}}}
\end{figure}

One observes that at each turning point in Figure \ref{fig6a}, exactly
one component of the symbol changes. For example, the symbol $(1,1)$
changes at a left turning point into the symbol $(2,1)$. Due to the
combination of humps with finite length, a perturbed hilltop bifurcation
\textit{decides}, whether $(1,1)$ bifurcates into the symbol $(2,1)$ or into
$(1,2)$, see Figure \ref{fig5}.  

For two-humped orbits, the system \eqref{hilltop2}
is a set of two equations in three variables $\lambda,
\tau_1, \tau_2$, called the hilltop normal form, cf.\ \cite{gs85}
\begin{equation} \label{hilltop2rep}
\lambda = \tau_1^2,\quad \lambda = \tau_2^2.
\end{equation}
Figure \ref{fig5} (left) shows the solution curves of \eqref{hilltop2rep}
while the red curves in  Figure \ref{fig5} (right) indicate
 the generic solution picture of a perturbed equation. 
Here we neglect more detailed bifurcation diagrams
which take into account small 
hysteresis effects w.r.t.\ the parameter $\lambda$,
see  \cite[Ch. IX, \S3]{gs85} for the unfolding theory.

\begin{figure}[hbt]
  \begin{center}
    \psfrag{l2}{2}
    \psfrag{l1}{1}
    \psfrag{a}{$(2,1)$}
    \psfrag{b}{$(1,1)$}
    \psfrag{c}{$(1,2)$}
    \psfrag{d}{$(2,2)$}
    \psfrag{t1}{$\tau_1+\tau_2$}
    \psfrag{t2}{$\tau_1-\tau_2$}
    \psfrag{l}{$\lambda$}
    \epsfig{width=12cm, file=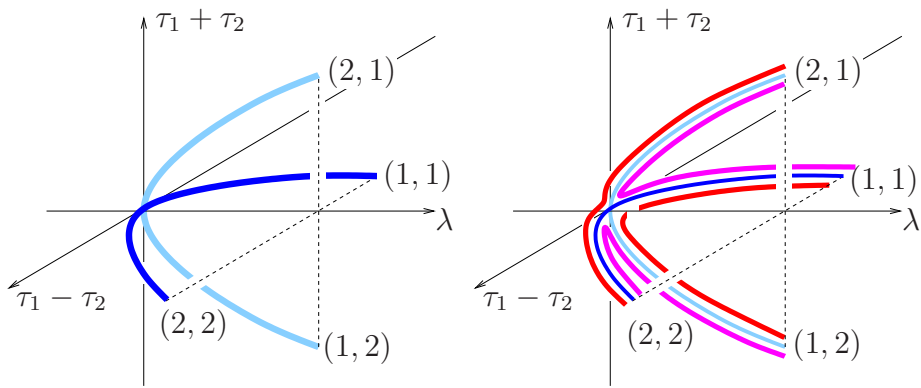}
  \end{center}
  \caption{\textit{Unperturbed (left) and perturbed
      hilltop-bifurcation (right) at the turning point 
      $\ell_{1,2}$.\label{fig5}}}
\end{figure}

\subsection{Connected components and equivalent symbols }
Homoclinic orbits that lie on a common 
closed curve define a connected component of 
$$
\mathcal{H} := \big\{(y_\Z,\lambda)\in \ell^\infty(\R^k)\times \R:
y_{n+1} = f(y_n,\lambda)\ \forall n \in \Z, \lim_{n\to \pm \infty} y_n
= \xi(\lambda)\big\}.
$$
More precisely, let  $s \in \Sn{n}$ and denote 
by $C(s)\subset \mathcal{H}$ 
the connected component that satisfies $(x_\Z(s),\tilde \lambda) \in
C(s)$.  

Then, we obtain an equivalence relation by 
identifying two sequences $s, \bar s \in \Sn{n}$, if the
corresponding orbits lie in the same component i.e.\
\begin{equation}\label{eq}
s\cong \bar s \quad  \Leftrightarrow \quad (x_\Z(\bar s),\tilde
\lambda) \in C(s). 
\end{equation}

In the following, we discuss how to find 
these equivalence classes. Particularly, we show
under some generic assumptions that each equivalence class has at
least four elements, and for $n$-humped orbits it turns out that one class
has at least $4n$ elements. 

For this task, we introduce a labeled graph $G$ with vertices
$s\in\Sn{n}$. Two vertices $s$ and $\bar s \in \Sn{n}$ are connected
with an $L$ or $R$-edge, if $x_\Z(s)$ bifurcates into $x_\Z(\bar s)$
via a left or right turning point. Since we do not know the effect
of the perturbed hilltop bifurcation a priori, we put an
edge, if the transition is possible for at least one perturbation.
For example, the
vertices $(1,1)$ and $(2,1)$ as well as $(1,1)$ and $(1,2)$ are
connected with $L$-edges in case $n=2$, see Section \ref{S41}.
Precise rules for constructing this graph are stated in Section
\ref{S43}.

Our hypothesis is that the desired equivalence classes
correspond to \textit{a special} decomposition of this graph
into disjoint $LR$-cycles.

In case of one-humped orbits, the only $LR$-cycle is $01230$, see
Figure \ref{fig10}. Consequently, all symbols lie in the same
equivalence class, which  matches the fact that all one-humped orbits lie
on the same closed curve and thus, in the same connected component of
$\mathcal{H}$. 

\subsection{Graph structure of homoclinic network}\label{S43}
In this section, we give precise rules for defining the labeled graph
$G$ which we identify with its adjacency tensor with entries $R$ and $L$.

First, we assume that only one of the $n$ humps can turn into a
neighboring hump at a turning point.
\begin{itemize}
\item [\textbf{R1}] There is no edge from $s\in \Sn{n}$ to $\bar s \in
  \Sn{n}$ if
$$
s = \bar s \quad \text{or}\quad 
\|s-\bar s\|_1 = \sum_{i=1}^nd(s_i,\bar s_i)  \ge 2,
$$
where $d$ is the distance on the cycle $01230$.
\end{itemize}

Now let $s, \bar s \in \Sn{n}$ and assume $\|s-\bar s\|_1 = 1$, 
then there exists 
a unique $j$ such that $s_j \neq \bar s_j$.

From $\lambda^{2,3} > \lambda^{0,1}$ we conclude that  
the right transition at $r_{2,3}$ can only occur   
if the orbit contains no $0$ and no $1$ hump. Therefore, we define $R$-edges in
$G$ according to the following rule.
\begin{itemize}
\item [\textbf{R2}] $G(s,\bar s) = R$ if
  \begin{itemize}
  \item [$\bullet$] $\{s_j, \bar s_j\} = \{0, 1\}$,
  \item [$\bullet$] $\{s_j, \bar s_j\} = \{2, 3\}$ and 
    $s_i \in \{2,3\}$ for all $i = 1,\dots,n$.
  \end{itemize}
\end{itemize}

Similarly from $\lambda^{1,2} > \lambda^{3,0}$ we conclude that the
left transition at $\ell_{3,0}$ can only occur if the orbit contains no
 $1$ and no $2$ hump. Our rules for $L$-edges are:
\begin{itemize}
\item [\textbf{R3}] $G(s,\bar s) = L$ if
  \begin{itemize}
  \item [$\bullet$] $\{s_j, \bar s_j\} = \{1, 2\}$,
  \item [$\bullet$] $\{s_j, \bar s_j\} = \{0, 3\}$ and 
    $s_i \in \{0,3\}$ for all $i = 1,\dots,n$.
  \end{itemize}
\end{itemize}

We expect a connected component to correspond to an $LR$-cycle in this
graph, i.e. a cycle on which $L$ and $R$-edges alternate. A precise statement
of our hypothesis is as follows. 

\begin{hyp}\label{hyp1}
The connected components of $n$-humped orbits and thus the
equivalence classes \eqref{eq} are in one to one correspondence to a 
partition of $G$ into disjoint $LR$-cycles.  
\end{hyp}

In case $n=2$, the $L$ and $R$-edges are shown in the left and center
picture of Figure \ref{fig6}. The right diagram additionally shows the
$LR$-cycles that correspond to the connected components from Figure
\ref{fig6a}. 

Similar diagrams for $n=3$ are shown in Figure \ref{fig7}.

\begin{figure}[hbt]
  \begin{center}
    \psfrag{l0}{$0$}
    \psfrag{l1}{$1$}
    \psfrag{l2}{$2$}
    \psfrag{l3}{$3$}
    \epsfig{width=15cm, file=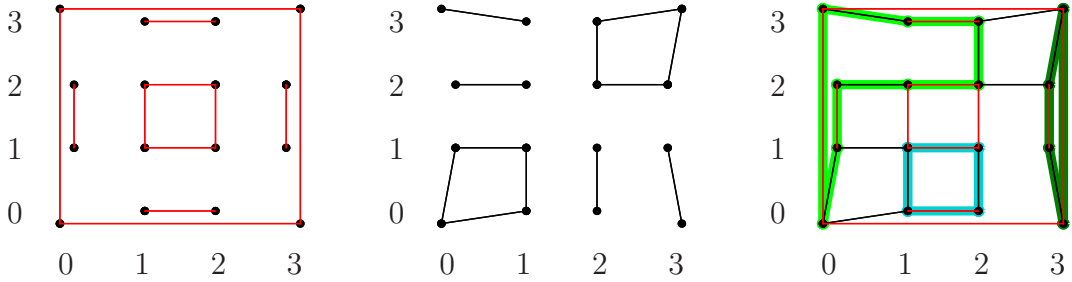}
  \end{center}
  \caption{\textit{$L$-edges of $G$ (left) and $R$-edges (center) for
      two-humped orbits. The
      cycles in the right figure correspond to the closed curves that are
      computed numerically in Figure \ref{fig6a}.  \label{fig6}}}
\end{figure}

\begin{figure}[hbt]
  \begin{center}
    \psfrag{000}{$000$}
    \psfrag{130}{$130$}
    \psfrag{430}{$430$}
    \psfrag{203}{$203$}
    \epsfig{width=14cm, file=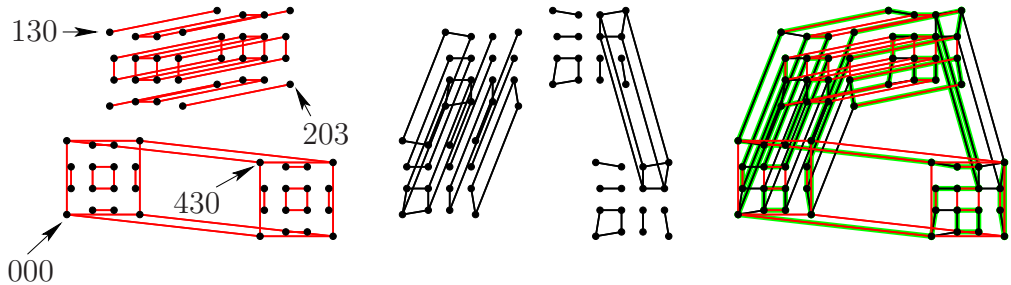}
  \end{center}
  \caption{\textit{$L$-edges of $G$ (left) and $R$-edges (center) for
      three-humped orbits. The
      green lines in the right figure show the cycles that are
      computed numerically.  \label{fig7}}}
\end{figure}

We continued $n$-humped orbits numerically for H\'enon's map up to
$n=5$. Table \ref{tab1} summarizes the number of cycles and their lengths
found in the computation.

\begin{table}[hbt]
\begin{center}
\begin{tabular}{c||c|c|c|c|c}
 &\multicolumn{5}{c}{length of cycle}\\
$n$ & 4 & 8 & 12 & 16 &20 \\\hline\hline
1& 1 \\\hline
2& 2 & 1\\\hline
3& 9 & 2 & 1\\\hline
4& 45 & 3 & 3 & 1\\\hline
5& 205 & 6 & 6 & 4 &1
\end{tabular}
\end{center}
 \caption{\textit{Continuation of $n$-humped orbits -- length of
     occurring cycles in numerical experiments. \label{tab1}}} 
\end{table}

These experiments for $n$-humped orbits of H\'enon's map 
suggest that the length of
occurring cycles is a multiple of $4$. Furthermore one cycle exists of
at least length $4n$.

\begin{theorem}\label{P1}
Fix $n \in \N$ and assume that Hypothesis \ref{hyp1} holds true.
Then, all $n$-humped orbits 
lie on  cycles whose length is a multiple of $4$. 

For the specific symbols
$s^0 = (0,\dots,0)$ and $s^2=(2,\dots,2) \in \Sn{n}$, the
corresponding orbits 
$x_\Z(s^0)$ and $x_\Z(s^2)$ lie on a common cycle 
of at least length $4n$. 
\end{theorem}

\begin{remark}
Table \ref{tab1} shows that there is exactly one orbit of length $4n$
and all other orbits are shorter. Hence, the orbit of length $4n$
contains $s^0$ and $s^2$.  
\end{remark}

\begin{proof}
By assuming Hypothesis \ref{hyp1} we see that it suffices to analyze
$LR$-cycles of the graph $G$. More precisely, 
we prove Theorem \ref{P1} along the following steps.
\begin{itemize}
\item [(i)] Each $LR$-cycle in $G$ has length $4m$, $m \ge 1$, $m \in \N$.
\item [(ii)] There exists an $LR$-cycle from $s^0$ to $s^2$ of length $4n$
  and each $LR$-cycle that contains $s^0$ and $s^2$ has at least length $4n$. 
\item [(iii)] Each $LR$-cycle that contains $s^0$ also contains $s^2$.
\end{itemize}

\begin{itemize}
\item [(i)] Let the vertices $v^1,\dots,v^m, v^{m+1}=v^1$ form an $LR$-cycle in
  $G$. 
Fix $j\in\{1,\dots,n\}$
  and let 
$
l_j = \#\big\{i\in\{1,\dots,m\} :v^i_j \neq v^{i+1}_j,\ G(v^i,v^{i+1}) =
L\big \}
$
be the number of $L$-edges for which the corresponding vertices only
differ in the $j$th component.

From \textbf{R1} it follows that $l_j$ is an even number and 
$ \max_{j=1,\dots,n} \{l_j\}\ge 2$, otherwise the cycle cannot be
closed in the $j$th component. Furthermore, an $LR$-cycle has the
same number of $L$ and $R$-edges. 
Thus, the cycle has length
 $$
  m=\sum_{i=1}^n 2l_i = 4 \sum_{i=1}^n \frac{l_i}2 = 4p
  \quad \text{with}\quad p = \sum_{i=1}^n \frac{l_i}2 \ge 1.
  $$
\item [(ii)] We explicitly construct an $LR$-cycle in $G$ from $s^0$ to $s^2$:
$(0\cdots 0) \mapsto (10\cdots 0)$ $ \mapsto (20\cdots 0) \mapsto (210\cdots
0)\mapsto (220\cdots 0)\mapsto \cdots \mapsto (2\cdots 2) \mapsto (32\cdots 2)
\mapsto (312\cdots 2)\mapsto (302\cdots 2)\mapsto (3012\cdots 2)\mapsto
(3002\cdots 2)\mapsto (30\cdots 0) \mapsto (0\cdots 0)$ which has length
$4n$.
Note that the distance from  $s^0$ to $s^2$ on the full quadratic grid
is $2n$ and consequently, each cycle containing these two points has at
least length $4n$. 
\item[(iii)] For proving that each $LR$-cycle with $s^0$ also contains
  $s^2$, assume that an $LR$-cycle exists that contains $s^2$ but not
  $s^0$. This cycle lies in the subgraph that we obtain by deleting 
  the vertex $s^0$ with its corresponding edges.

  Note that the $LR$-cycles start with an $L$ and end with an
  $R$-edge, and it follows from \textbf{R2} that all $R$-edges that
  start at $s^2$ end in  
  $$
  V_3:=\{s\in \Sn{n}:\ \exists
  j\in\{1,\dots,n\}:\ s_j=3\}.
  $$
  We define the graph $\tilde G$ by deleting the $R$-edges of
  $s^2$ from the remaining graph. Figure \ref{fig8} illustrates this
  construction in case $n=2$.

\begin{figure}[hbt]
  \begin{center}
    \psfrag{l0}{$0$}
    \psfrag{l1}{$1$}
    \psfrag{l2}{$2$}
    \psfrag{l3}{$3$}
    \psfrag{s0}{$s_0$}
    \psfrag{s2}{$s_2$}
    \epsfig{width=6cm, file=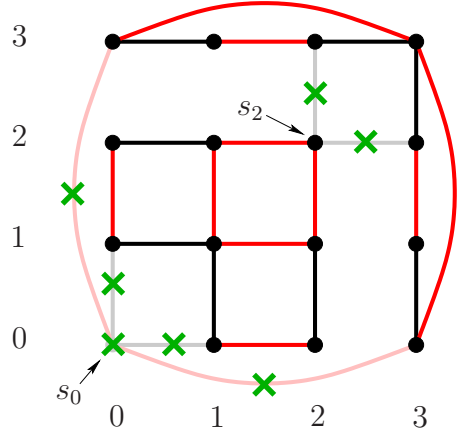}
  \end{center}
  \caption{\textit{Transition graph in case $n=2$. $L$-edges
      and $R$-edges are plotted in red and black, respectively. Vertices
      and edges that are deleted in the proof of Theorem 
      \ref{P1} are marked by green crosses. \label{fig8}}}
\end{figure}

If $\tilde G$ breaks into two components 
$V_3$ and $\tilde G\setminus V_3$, then we get a contradiction to the
above assumption and an
$LR$-cycle in the original graph $G$ that contains $s^2$ but not $s^0$
cannot exist.

To finish the proof, we show that an $LR$-path in $\tilde
G$ from $s^2$ to $V_3$ does not exist.

From $s^2$ we cannot go directly 
via an $R$-edge to $V_3$, since the corresponding
edges are deleted in $\tilde G$.
Thus without loss of generality, we get:
$(2\cdots 2) \mapsto (12\cdots 2) \mapsto (02\cdots 2)$.
Denote by $v^m$ the $m$th vertex on this path.
A $3$-component can only be achieved by the left transition $\ell_{3,0}$ or by
the right transition $r_{2,3}$, see \textbf{R2} and \textbf{R3}. 

If a $j$ exists with $v^m_j \in\{0,1\}$, then the $r_{2,3}$
transition is impossible by \textbf{R2}.

If a $j$ exists with $v^m_j \in\{1,2\}$,  then the $\ell_{3,0}$
transition is impossible by \textbf{R3}.

Thus, we only obtain a $3$ component via the vertex $s^0=(0\dots 0)$ which is
deleted in $\tilde G$.
\end{itemize}
\end{proof}

\section{Bifurcation analysis near homoclinic tangencies}
\label{S5}
In this section we prove  the main Theorem \ref{Th1.1} and Theorem \ref{Th1.3}(i)
by using an existence and uniqueness result for a suitable operator equation
in spaces of bounded sequences. In the following we use the notation
 $B_{\rho}(x)$ and $B_\rho = B_\rho(0)$ to denote closed balls of radius $\rho$ in some
Banach space.

\subsection{The operator equation}
\label{ss5.1}

First recall the operator $F:\ell^\infty(\R^k)\times \R \to
\ell^\infty(\R^k)$ from \eqref{Fdef} and the normalization 
$\bar{\lambda}=0$ and $\xi(\lambda)=0$ for $\lambda$ close to $\bar{\lambda}$,
see \textbf{A2}, \textbf{A3} in Section \ref{S2}.
Then for any $s \in \Omega_N$ define the operator 
\begin{equation*} 
G_s: \ell^\infty(\R^k) \times \ell^\infty(s)\times \ell^\infty(s)
\times \R \to \ell^\infty(\R^k) \times \ell^\infty(s)
\end{equation*}
by
\begin{equation} \label{defop}
G_s(x_\Z,g,\tau,\lambda)  =
\begin{pmatrix}
F(p_\Z(s)+x_\Z+v_\Z(s,\tau),\lambda) + w(s,g)\\
\langle \beta^{-\ell} u_\Z, x_\Z\rangle,\ \ell \in I(s)  
\end{pmatrix}.
\end{equation}
Here $p_\Z$, $v_\Z$ are defined in \eqref{1.16}, \eqref{1.17} and
$w(s,g)$ is given by (recall $w_{\Z}$ from \eqref{1.22})
$$
w(s,g) = \sum_{\ell \in I(s)} g_\ell \beta^{-\ell} w_\Z,\quad g \in
\ell^\infty(s).
$$
Our aim is to derive the functions $x_{\Z,s}$, $g_s$ in Theorem
\ref{Th1.1} by solving
\begin{equation}\label{5.4}
G_s(x_\Z,g,\tau,\lambda) = 0
\end{equation}
for $\|\tau\|_\infty$, $|\lambda|$ sufficiently small and for all 
$s \in \Omega_N$. 
 More precisely,
we prove in Section \ref{S6} the following \textit{Reduction Theorem}.
\begin{theorem}\label{Th5.1}
There exist constants  $C_0,\rho_x,\rho_g,\rho_{\tau},\rho_{\lambda} > 0$ 
and a number $N_0 \in \N$ such 
that for all $N\ge N_0$ and for all $s \in \Omega_N$
the following statements hold. 
For all $\tau \in B_{\rho_\tau}$, $\lambda \in B_{\rho_\lambda}$ the
system \eqref{5.4} has a unique solution
$$
g = g_s(\tau,\lambda) \in B_{\rho_g} \subset \ell^{\infty}(s), \quad 
x_\Z = x_{\Z,s}(\tau,\lambda) \in B_{\rho_x}\subset \ell^{\infty}(\R^k).
$$
Moreover, the following estimate is satisfied:
\begin{equation} \label{Ginvest}
\max(\|g-\tilde g\|_\infty, \|x_\Z-\tilde x_\Z\|_\infty) \le C_0
\|G_s(x_\Z,g,\tau,\lambda) - G_s(\tilde x_\Z,\tilde
g,\tau,\lambda)\|
\end{equation}
for all $g,\tilde g \in B_{\rho_g}$, $x_\Z,\tilde x_\Z \in
B_{\rho_x}$, $\tau \in B_{\rho_\tau}$, $\lambda \in B_{\rho_\lambda}$.
\end{theorem}

\subsection{Preparatory Lemmata}
\label{ss5.2}
In order to construct the neighborhoods $U$ and $\Lambda$ in Theorem
\ref{Th1.1} we need several lemmata.

\begin{lemma}\label{L5.2}
There exists $N_1 \in \N$, such that for all $N \ge N_1$, $s \in
\Omega_N$ the linear system
\begin{equation}\label{5.6}
\sum_{k \in I(s)} \langle \beta^{-\ell} u_\Z,\beta^{-k} u_\Z\rangle
\tau_k = r_\ell,\quad \ell \in I(s),\ r \in \ell^\infty(s)
\end{equation}
has a unique solution $\tau \in \ell^\infty(s)$ and
\begin{equation}\label{5.7}
\|\tau\|_\infty \le 2 \|r\|_\infty.
\end{equation}
\end{lemma}

\begin{proof}
Rewrite \eqref{5.6} as fixed point equation
$$
\tau = P \tau + r,\quad 
(P\tau)_\ell = - \sum_{k\in I(s), k\neq \ell} \langle \beta^{-\ell} u_\Z,
\beta^{-k} u_\Z\rangle \tau_k
$$ 
and note
\begin{eqnarray*}
\|P\tau\|_\infty &\le &
\|\tau\|_\infty \sup_{\ell \in I(s)} \sum_{k\in I(s), k\neq \ell}
|\langle u_\Z, \beta^{\ell -k} u_\Z\rangle |\\
&\le & \|\tau\|_\infty C \sum_{j \ge 1} \e^{-\alpha j N} 
= \frac {C \e^{-\alpha N}}{1-\e^{-\alpha N}} \|\tau\|_\infty.
\end{eqnarray*}
Thus, we choose $N_1$ such that $\frac {C \e^{-\alpha
    N_1}}{1-\e^{-\alpha N_1}}\le \frac 12$. Then $P$ is contractive and
\eqref{5.7} follows. 
\end{proof}

\begin{lemma}\label{L5.3}
There exist $N_2 \in \N$, $C_2 > 0$ such that 
\begin{equation}\label{5.9}
\|F(p_\Z(s) + v_\Z(s,\tau),\lambda)\|_\infty \le C_2(|\lambda| +
\|\tau\|_\infty^2 + \e^{-\alpha N/2})
\end{equation}
for all $N \ge N_2$, $s\in \Omega_N$, $|\lambda| \le 1$, $\tau \in B_1
\subset \ell^\infty(s)$.
\end{lemma}

\begin{proof}
We estimate
\begin{eqnarray*}
&& p_{n+1}(s) + v_{n+1}(s,\tau) - f(p_n(s)+v_n(s,\tau),\lambda)\\
&=& \sum_{\ell \in I(s)}(\bar x_{n+1-\ell} + \tau_\ell u_{n+1-\ell}) -
f(p_n(s) + v_n(s,\tau),0) + \OO(|\lambda|)\\
&=& \sum_{\ell \in I(s)} \big( f(\bar x_{n-\ell},0) + \tau_\ell
  f_x(\bar x_{n-\ell},0) u_{n-\ell}\big)\\
&& - f\big(\sum_{\ell \in
    I(s)} \bar x_{n-\ell} + \sum_{\ell \in I(s)} \tau_\ell u_{n-\ell}
  ,0\big) +\OO(|\lambda|)\\
&=& \sum_{\ell \in I(s)} f(\bar x_{n-\ell},0) +  \sum_{\ell \in I(s)}
\tau_\ell f_x(\bar x_{n-\ell},0) u_{n-\ell} \\
&& - f\big(\sum_{\ell \in I(s)} \bar x_{n-\ell},0\big) 
- f_x\big(\sum_{j \in I(s)} \bar x_{n-j},0\big) \sum_{\ell \in
  I(s)} \tau_\ell u_{n-\ell} + \OO(|\lambda| + \|\tau\|_\infty^2).
\end{eqnarray*}
For $n\in\Z$, choose $\tilde \ell \in I(s)$ such that $|n-\tilde \ell|
\le |n-\ell|$ for all $\ell \in I(s)$. Then $|n-\ell|\ge \frac{N}{2}$ for
all $\ell \neq \tilde{\ell}$ and, therefore, by \eqref{exdec},
\begin{equation} \label{pseudofest}
\begin{aligned}
&&\Big\|\sum_{\ell \in I(s)} f(\bar x_{n-\ell},0) - f\big(\sum_{\ell \in
  I(s)} \bar x_{n-\ell},0\big)\Big\|\\
&\le& \Big\|\sum_{\ell \in I(s),\ell \neq \tilde \ell} f(\bar
x_{n-\ell},0)\Big\| + \Big\|f(\bar x_{n-\tilde \ell},0) -
f\big(\sum_{\ell \in I(s)} \bar x_{n-\ell},0\big) \Big\|\\
&\le & \Big\| \sum_{\ell \in I(s),\ell \neq \tilde \ell} \bar
x_{n+1-\ell}\Big\| + L  \Big\| \sum_{\ell \in I(s),\ell \neq \tilde
  \ell} \bar x_{n-\ell}\Big\| \le C_2 \e^{-\alpha N /2}.
\end{aligned}
\end{equation}
In a similar way, 
\begin{equation} \label{pseudouest}
\sum_{\ell \in I(s)} \tau_\ell f_x\big(\sum_{j \in I(s)} \bar
x_{n-j},0\big)u_{n-\ell} = 
\sum_{\ell \in I(s)} \tau_\ell f_x(\bar
x_{n-\ell},0)u_{n-\ell} + \OO(\e^{-\alpha N/2}\|\tau\|_{\infty}). 
\end{equation}
Combining these estimates, we obtain \eqref{5.9}.
\end{proof}

\begin{lemma}\label{L1}
Assume \textbf{A1}, \textbf{A2} and let $\bar x_\Z$ be a homoclinic
$f(\cdot, 0)$-orbit with respect to the hyperbolic fixed point $0$.
Then, there exist zero neighborhoods $U_3\subset \R^k$,
$\Lambda_3 \subset \Lambda_0$ and constants $N_3, n_0 ,\alpha>0, C_3\ge 1$
 such that the following statement holds for all $K \ge N_3$,
$-n_-,n_+ \ge n_0$:

If $x_{n+1} = f(x_n,\lambda)$ for $n \in \tilde J:=[0,K-1]$, 
$\lambda\in \Lambda_3$, and if $x_n
\in U_3$ for all $n \in J:=[0,K]$ 
then we have the estimate
$$
\sup_{j \in J} \|x_j - \bar x_{n_-+1-K+j} - \bar x_{n_+-1+j}\|
\le C_3 \left(\|x_K - \bar x_{n_-+1}\| + \|x_{0} - \bar x_{n_+-1}\| + |\lambda|+
  \e^{-\alpha \frac K2}\right).
$$ 
Furthermore we obtain in case $K = \infty$:
\begin{equation}\label{unendl}
\sup_{j \ge 0} \|x_j - \bar x_{n_+-1+j}\| \le C_3 (\|x_0 - \bar
  x_{n_+-1}\| + |\lambda|).
\end{equation}
\end{lemma}


\begin{proof}
Consider the pseudo-orbit $p_j:= \bar x_{n_+-1+j} + \bar x_{n_-+1-K+j}$,
$j\in J$ which is a zero of the boundary value operator 
$$
\Gamma_J(y_J,\lambda) :=
\begin{pmatrix}
y_{n+1} - f(y_n, \lambda)-\rho_n,\quad n \in \tilde J\\
b_K(y_{0},y_K)  
\end{pmatrix},
$$
at $\lambda= 0$, where
\begin{eqnarray*}
\rho_n &:=& f(\bar x_{n_+-1+n},0) + f(\bar x_{n_-+1-K+n},0) - 
f(p_n,0),\quad n \in J,\\
b_K(y_{0},y_K) &:=&
\begin{pmatrix}
P_s(y_{0} - p_{0})\\ P_u(y_K-p_K)  
\end{pmatrix}.
\end{eqnarray*}
Here $P_s$ and $P_u$ are the stable and unstable projectors of the
fixed point $0$.


Let $b$ be the bound from Theorem
\ref{rough} for the difference equation 
$$
u_{n+1} = (f_x(0,0) + B_n)u_n,\quad B_n = f_x(p_n,\lambda) -
f_x(0,0),\quad n\in \tilde J.
$$
For sufficiently large $-n_-, n_+ \ge n_0$ and $\lambda \in \Lambda_3$
sufficiently small, we get $\|B_n\| \le b$ for all $n \in \tilde
J$. Consequently 
$$
u_{n+1} = f_n(p_n,\lambda) u_n ,\quad n \in \tilde J
$$
has an exponential dichotomy on $J$ with projectors $P_n^s,P_n^u$ and an exponential rate $\alpha$
that is independent of $n_-$, $n_+$, $\lambda$ and $K$. 

As in the proof of \cite[Theorem 4]{hu11}, we show that for
$\lambda \in \Lambda_3$, $n_-, n_+ \ge n_0$ and $K \ge N_3$ we have a
uniform bound 
\begin{equation}\label{inverseest}
\|D_1\Gamma_J(p_J,\lambda)^{-1}\|_\infty \le \sigma^{-1}.
\end{equation}

In order to see this, consider the inhomogeneous difference equation 
\begin{eqnarray}
u_{n+1} - f_x(p_n,\lambda) u_n &=& r_n,\quad n = 0,\dots K-1,\label{inhom1}\\
P_s u_0 + P_u u_K &=& \gamma.  \label{inhom2}  
\end{eqnarray}
Denote by $\Phi$ the solution operator of the homogeneous equation and
let $G$ be the corresponding Green's function, cf.\ \eqref{greenplus}. 
The general solution of \eqref{inhom1} is given by
\begin{equation}\label{lsg1}
u_n = \Phi(n,0) v + \sum_{m \in \tilde J} G(n,m+1) r_m,
\end{equation}
where
$$
v = v_- + \Phi(0,K) v_+,\quad v_- \in \range(P_0^s),\quad v_+ \in
\range(P_K^u).
$$ 
Inserting \eqref{lsg1} into \eqref{inhom2}, it remains to solve
$$
P_s \big(v_- + \Phi(0,K)v_+\big) + 
P_u \big(\Phi(K,0)v_- + v_+\big) = R,
$$
with 
$$
R = \gamma - P_s \sum_{m\in \tilde J} G(0,m+1)r_m 
 - P_u  \sum_{m\in \tilde J} G(K,m+1)r_m. 
$$
This finite-dimensional system has a unique solution for $K\ge N_3$ sufficiently large
since $\|P_s - P_0^s\| \to 0$ and $\|P_u-P_K^u\| \to 0$ as $K\to \infty$.
Therefore, the system \eqref{inhom1}, \eqref{inhom2} also has a unique solution $u_J$
for $K$ large and the dichotomy estimates lead to a bound
$$
\|u_J\|_\infty \le \sigma^{-1} (|\gamma| + \|r_J\|_\infty),
$$
i.e. \eqref{inverseest} holds.

We apply Theorem \ref{Va2} with the space 
$Y=\ell^{\infty}_{J}=\{(y_n)_{n\in J}: y_n \in \R^k \}$ of finite sequences 
and with $Z=\ell^{\infty}_{\tilde{J}}\times \R^k$,  both endowed with the sup-norm. 
We take $y_0 =p_J$ and use uniform data for all $\lambda \in \Lambda_3$. 
For $\delta$ sufficiently small we have 
$$
\|D_1\Gamma_J(x_J,\lambda) - D_1\Gamma_J(p_J, \lambda)\|_\infty \le
\frac \sigma 2 \text{ for all sequences} \; x_J \in 
B_\delta(p_J), \lambda \in \Lambda_3,
$$
and by choosing the neighborhood $\Lambda_3$ sufficiently small we get 
\begin{eqnarray*}
\|\Gamma_J(p_J,\lambda)\|_\infty &=& \sup_{n\in \tilde J}\|p_{n+1} -
f(p_n,\lambda) - \rho_n\| + \|b_K(p_{0},p_{K})\|\\ 
&=& \sup_{n\in\tilde J}\| p_{n+1} - f(p_n,\lambda) + f(p_n,0)\\
&&\qquad  - f(\bar x_{n_+-1 + n},0) - f(\bar x_{n_-+1-K + n},0)\| \\
&=& \sup_{n\in\tilde J}\|f(p_n,0) - f(p_n,\lambda)\| \le \frac \sigma 2 \delta
\quad \text{for } \lambda \in \Lambda_3.
\end{eqnarray*}
Theorem \ref{Va2} applies to $\lambda \in \Lambda_3$ with uniform data, 
and  it follows from \eqref{A2} with some constant $C_3 > 0$ that
\begin{equation}\label{A5}
\|x_J - y_J\|_{\infty} \le C_3 
\|\Gamma_J(x_J,\lambda) - \Gamma_J(y_J,\lambda)\|_{\infty}
 \quad \text{for
  all } x_J, y_J \in B_\delta(p_J), \lambda\in \Lambda_3.
\end{equation}
From \eqref{exdec} we find a number $n_0$ such that
 $\bar x_n \in B_{\frac \delta 2}(0)$ for all $|n|\ge n_0$ 
and also $p_n\in B_{\frac \delta 2}(0), n\in J$
for all $-n_-,n_+ \ge n_0$. Then we take $U_3 := B_{\frac {\delta} {2}}(0)$ 
as our neighborhood and note that
\eqref{A5} holds for any two sequences $x_J, y_J$ in $U_3$.
 For $n\in \tilde J$ and $\lambda \in \Lambda_3$ 
it follows that 
\begin{eqnarray*}
&&f(p_n,\lambda) - p_{n+1}\\ 
&=& f(\bar x_{n_+ -1 + n} + \bar x_{n_- +1 -K + n} , \lambda)
- \bar x_{n_+ -1 + n + 1} - \bar x_{n_- +1 -K + n + 1} \\
&=& f(\bar x_{n_+ -1 + n} + \bar x_{n_- +1 -K + n} ,0)
-f(\bar x_{n_+ -1 + n },0) - f(\bar x_{n_- +1 -K + n},0) + \OO(|\lambda|)\\
&=&
\left\{
\begin{array}{ll}
- f(\bar x_{n_- +1 -K + n},0) + \OO(\|\bar x_{n_-+1-K+n}\|)
&\text{ for } 0\le n \le \frac K2\\[2mm]
- f(\bar x_{n_+ -1 +n},0)+ \OO(\|\bar x_{n_+ -1 + n}\|) &
\text{ for }\frac K2 < n \le K  
\end{array}\right\}
+ \OO(|\lambda|)\\
&=& \OO(\e^{-\alpha \frac K2} + |\lambda|).
\end{eqnarray*}

Now let $x_J$ be a sequence in $U_3$ such that $x_{n+1} = f(x_n,\lambda)$ for all $n \in
\tilde J$, and some $\lambda \in \Lambda_3$. Then 
\begin{eqnarray*}
\|x_J - p_J\|_\infty &\le & C \|\Gamma_J(x_J, \lambda) -
\Gamma_J(p_J,\lambda)\|_\infty \\
&\le& C \Big(\sup_{n\in\tilde J} \|f(p_n,\lambda)-p_{n+1}\| +
  \|b_K(x_{0},x_K)\|\Big) \\
&\le & C \left(\e^{-\alpha \frac K2} +   |\lambda|+ \left\|
  \begin{pmatrix}
   P_s(x_{0} - p_{0}) \\ P_u(x_K - p_K)  
  \end{pmatrix}
  \right\| \right)\\
&\le& C \left( \e^{-\alpha \frac K2} + \|x_{0} -p_{0}\| + \|x_K - p_K\| +|\lambda| 
  \right)\\
&\le& C_3\left( \e^{-\alpha \frac K2} + \|x_0 - \bar x_{n_+-1}\|
  +\|x_K - \bar x_{n_-+1}\| +|\lambda|\right).
\end{eqnarray*}
In case $K = \infty$, one uses the operator 
$$
\tilde \Gamma_\N(y_\N,\lambda) =
\begin{pmatrix}
y_{n+1} - f(y_n),\quad n \in \N\\
P_s(y_0-\bar x_{n_+-1})  
\end{pmatrix},
$$
and it turns out that $D_1\tilde \Gamma_\N(p_\N,\lambda)$ with 
$p_j = \bar x_{n_+-1+j}$ has a uniformly bounded inverse
for $\lambda \in \Lambda_3$ and
sufficiently large $-n_-,\ n_+$. Then the estimate \eqref{unendl}
follows immediately. 

\end{proof}
\subsection{Proof of Main Theorem}
Let us first prove assertion (i) in Theorem \ref{Th1.1}.\\
\label{ss5.3}
{\bf Step 1:} (Construction of neighborhoods $U,\Lambda$) \\
In the following $\Lambda_1\supset \Lambda_2\supset \ldots$ will denote
shrinking neighborhoods of $0$. Let $\rho_g,\rho_x>0$ be given by
Theorem \ref{Th5.1} and note that we can decrease $\rho_{\tau},\rho_{\lambda}$
without changing the assertion of Theorem \ref{Th5.1}.
Introduce the constants (cf.\ \eqref{exdec} and  Lemma \ref{L5.3}, \ref{L1}) 
\begin{equation} \label{const1}
\alpha_*= \e^{-\alpha}, \quad C_4= C_3 +\frac{2C_e}{1-\alpha_*}, 
\quad C_5=C_e \frac{3 - \alpha_*}{1-\alpha_*}, \quad C_{6}=C_0 C_2+
\frac{2C_e}{1-\alpha_*}.
\end{equation}
Let $\rho_{\tau}>0$ be such that
\begin{equation} \label{rhotaudef}
C_5 \rho_{\tau} \le \frac{\rho_x}{4}.
\end{equation}
By Lemma \ref{L1} we can  choose a ball $B_{3\varepsilon_0} \subset U_3$ 
and numbers $n_+,-n_- \geq n_0$ such 
that  $\bar{x}_n \in B_{\varepsilon_0} $ for all $n\ge n_+ -1$ and 
$n \leq n_-+1$.
It is well known that the only full orbit in a small
neighborhood of a hyperbolic fixed point is the fixed point itself. That
is, we can assume w.l.o.g. that $U_3,\Lambda_3$ satisfy
\begin{equation} \label{smallhyp}
y_{n+1}=f(y_n,\lambda), \ y_n\in U_3 (n\in \Z),\ \lambda \in \Lambda_3
\Rightarrow y_{n}=0 \; \text{for all} \; n\in \Z.
\end{equation}
The set $\mathcal{K}=\{0\}\cup \{\bar{x}_n: n\leq n_{-} \; 
\text{or}\; n\geq n_+\} $ 
is compact and satisfies 
\[f(\mathcal{K},0)\subset \mathcal{K}\cup \{\bar{x}_{n_-+1}\}.
\]
Thus we find an $\varepsilon \le \varepsilon_0$  and 
$\Lambda_4 \subset \Lambda_3$ such that the following
properties hold
\begin{equation} \label{defV0}
V_0:=B_{\varepsilon}\cup \bigcup_{n\le n_-,n\ge n_+}B_{\varepsilon}(\bar{x}_n) 
\subset B_{2 \varepsilon_0} \subset U_3,
\end{equation}
the balls
\begin{equation} \label{defBj}
 B_j:=B_{\varepsilon}(\bar x_{n_-+j}),\ j=1,\ldots, \varkappa:=n_+-1-n_-
\end{equation}
are mutually disjoint,
\begin{equation} \label{V0exit}
f(V_0,\Lambda_4) \cap B_j = \emptyset \quad \text{for} \quad j=2,\ldots,
\varkappa,
\end{equation}
\begin{equation} \label{epsconst}
2C_5(2C_3+1)\varepsilon \le \frac{\rho_{\tau}}{3} \quad \text{and} \quad
(2C_3+1)\varepsilon \le \frac{\rho_x}{4},
\end{equation}
see Figure \ref{fig9}.
\begin{figure}[hbt]
  \begin{center}
    \psfrag{xi}{$0$}
    \psfrag{x-}{$\bar{x}_{n_-}$}
    \psfrag{x+}{$\bar{x}_{n_+}$}
    \psfrag{x2}{$\bar{x}_{n_-+1}$}
    \psfrag{x3}{$\bar{x}_{n_+-1}$}
    \psfrag{V0}{\textcolor{blue}{$V_0$}}
    \psfrag{V1}{\textcolor{blue}{$V_1$}}
    \psfrag{V2}{\textcolor{blue}{$V_2$}}
    \psfrag{VJJ}{\textcolor{blue}{$V_{\varkappa-1}$}}
    \psfrag{VJ}{\textcolor{blue}{$V_\varkappa$}}
    \psfrag{B1}{$B_1$}
    \psfrag{B2}{$B_2$}
    \psfrag{BJ}{$B_\varkappa$}    
    \psfrag{BJJ}{$B_{\varkappa-1}$}
    \epsfig{width=8cm, file=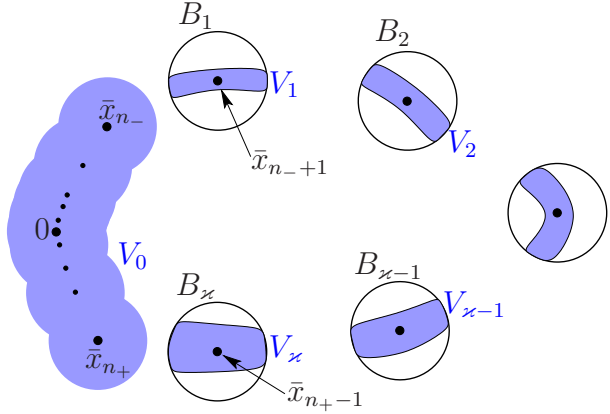}
  \end{center}
  \caption{\textit{Construction of neighborhoods. \label{fig9}}}
\end{figure}

Next we take $N_4\ge \max(N_1,N_2,N_3)$  (see Lemmata \ref{L5.2}, \ref{L5.3},
\ref{L1}) such that
\begin{equation} \label{Nconst}
2C_5C_4 \e^{-\alpha N_4} \le \frac{\rho_{\tau}}{3}, \quad
C_4 \e^{-\alpha N_4} \le \frac{\rho_x}{4}, \quad (2(C_e+C_3)+ 18C_3C_6)\e^{-\alpha N_4}\le
\frac{\varepsilon}{3}.
\end{equation}
We can also find $\Lambda_5 \subset \Lambda_4$ such that for all
$\lambda\in \Lambda_5$
\begin{equation} \label{lambdaconst}
2C_5C_3 |\lambda| \le  \frac{\rho_{\tau}}{3}, \quad 
 C_3 |\lambda| \le  \frac{\rho_x}{4}, \quad |\lambda|\le \rho_{\lambda},
\quad 6 C_3 |\lambda|(3C_0C_2+1) \le \varepsilon.
\end{equation}
Finally, we define
\begin{equation} \label{Ndef}
N= \varkappa + N_*, \quad \text{where} \quad N_*=2 N_4.
\end{equation}
Then we choose $\Lambda_6\subset \Lambda_5$  such that the following 
settings define neighborhoods $V_j$ of $\bar{x}_{n_-+j},j=\varkappa,\ldots,1$ 
recursively (cf.\ Figure \ref{fig9}):
\begin{equation} \label{vjdef}
\begin{aligned}
V_{\varkappa}&= B_{\varkappa}\cap \bigcap_{n=1}^{N_*+1}f^{-n}(V_0,\Lambda_6), \\
V_j & = B_j \cap f^{-1}(V_{j+1},\Lambda_6), \quad \text{for} \quad
j=\varkappa-1,\ldots,1.
\end{aligned}
\end{equation}
Here we use the notation  $f^{-n}(V_j,\Lambda_6)=
\{x:f^n(x,\lambda)\in V_j \; \text{for all} \; \lambda \in \Lambda_6\}$.

With these settings we consider the maximal invariant set $M(U,\Lambda)$,
cf.\ \eqref{maxinv}, that belongs to 
\begin{equation} \label{Ulamdef}
U = \bigcup_{j=0}^{\varkappa} V_j, \quad \Lambda = \Lambda_6.
\end{equation}
Let us note that our construction \eqref{defBj}, \eqref{V0exit}, 
\eqref{vjdef}, \eqref{Ulamdef}
 implies the following three
assertions for any $f(\cdot,\lambda)$-orbit $y_{\Z}\subset U$, $\lambda \in
\Lambda$
\begin{equation} \label{prop1}
y_{n}\in V_0, y_{n+1} \notin V_0 \Rightarrow y_{n+1} \in V_1,
\end{equation}
\begin{equation} \label{prop2}
y_n \in V_j \quad \text{for some}\quad 1\le j \le \varkappa-1 \Rightarrow
y_{n+1} \in  V_{j+1},
\end{equation}
\begin{equation} \label{prop3}
y_n\in V_{\varkappa} \Rightarrow y_{n+\ell}\in V_0 \quad \text{for} \quad
\ell=1,\ldots,N_*+1.
\end{equation}
{\bf Step 2:}(Construction of symbolic sequence $s$) \\
For some $\lambda \in \Lambda$ consider an orbit $y_{\Z}$ of \eqref{a1}
that lies in $U$. If it lies in $V_0$ then $y_n=0,n\in \Z$ by 
\eqref{smallhyp} and we set $s=0\in \Omega_N$.
Otherwise we have $y_{\tilde{n}}\notin V_0$ for some $\tilde{n}\in \Z$. 
We show that
\[ \tilde{I}(y_{\Z}):=\{ n \in \Z: y_n\in V_1 \}
\]
is nonempty and that there is a unique $s\in \Omega_N$ such that
$\tilde{I}(y_{\Z})= I(s)$, see \eqref{indexset}. By our assumptions
we have $y_{\tilde{n}}\in V_{j_0}$ for some $j_{0} \in \{1,\ldots,\varkappa\}$.
If $j_0=1$ then $\tilde{I}(y_{\Z})\neq \emptyset$ whereas in case
$j_0\ge2$ we obtain $y_{\tilde{n}-m}\in V_{j_0-m},m=0,\ldots,j_0-1$ by
induction from \eqref{prop2} and the fact that the $V_j$ are mutually disjoint.
Therefore, we have $\tilde{\ell}:=\tilde{n}-j_0+1 \in \tilde{I}(y_{\Z})$.
Moreover, from \eqref{prop1} and \eqref{prop3} we obtain 
\begin{equation} \label{ypos}
y_{\tilde{\ell}+j} \in V_{j+1}, j=0,\ldots \varkappa-1, \qquad
y_{\tilde{\ell}+j}\in V_0, j=\varkappa,\ldots, \varkappa+N_*.
\end{equation} This shows that
the difference of two consecutive indices $\tilde{\ell}< \ell$ in 
$\tilde{I}(y_{\Z})$ is at least $N=\varkappa+N_*$. Therefore $\tilde{I}(y_{\Z})$
belongs to $\Z(N)$ (cf.\ \eqref{zndef}) and there is a unique sequence
$s\in \Omega_N$ such that $\tilde{I}(y_{\Z})=I(s)$.

The relations \eqref{ypos} hold whenever $\ell\in I(s)$. By \eqref{defBj}
and \eqref{vjdef} this gives us the estimates
\begin{equation} \label{yinteriorest}
\|y_{p+\ell+\nu}-\bar{x}_p\| \le \varepsilon \quad \text{for} \quad
\ell\in I(s), p=n_-+1,\ldots,n_+-1
\end{equation} 
for the index $\nu=-n_--1$ (cf.\ \eqref{1.20}).
From this we will derive the estimate
\begin{equation} \label{serest}
T_n:=\|y_{n+\nu}- \sum_{m\in I(s)}\bar{x}_{n-m} \| \le C_4 \e^{-\alpha N_4}+C_3 |\lambda| +
(2C_3+1)\varepsilon \quad \text{for} \quad n\in \Z.
\end{equation}
Consider first indices $n=p+\ell$ with $\ell\in I(s)$ and 
$p=n_-+1,\ldots,n_+-1$. Then with \eqref{Ndef} we find
\begin{equation} \label{estmiddle} 
\begin{aligned}
 T_n 
& \le  \|y_{p+\ell+\nu}-\bar{x}_p\| +\sum_{m\in I(s),m\neq \ell} 
\|\bar{x}_{p+\ell-m}\| \\
& \le  \varepsilon + C_e \left(
\sum_{I(s)\ni  m<\ell}\alpha_*^{p+\ell-m}+\sum_{I(s)\ni m >\ell}
\alpha_*^{m-\ell-p}\right) \\
&\le   \varepsilon + C_e \left(
\sum_{\mu\ge 1}\alpha_*^{\mu N+n_-+1}+\sum_{\mu\ge1}\alpha_*^{\mu N -n_++1}\right) \\
& =  \varepsilon + \frac{C_e}{1-\alpha_*^N}\left( 
\alpha_*^{N+n_-+1}+ \alpha_*^{N-n_++1} \right) \\
& \le  \varepsilon + \frac{2C_e}{1-\alpha_*}\alpha_*^{N_*}.
\end{aligned}
\end{equation}
Next consider two consecutive indices $\ell < \tilde{\ell}$ in $I(s)$
and $n=p+\ell$ for $p=n_+-1,\ldots,\tilde{p}=\tilde{\ell}-\ell+n_-+1$.
For these indices we get  
$$
y_{n+\nu}\in 
\left\{
\begin{array}{rl}
V_\varkappa, & \text{ for } p = n_+-1,\\
V_1, & \text{ for } p = \tilde p,\\
V_0,& \text{ otherwise},    
\end{array}
\right.
$$
and we can apply Lemma \ref{L1} 
to this sequence in place of $\tilde{x}_0,\ldots,\tilde{x}_K$,
where  $K=\tilde{\ell}-\ell+2-(n_+-n_-)=\tilde{\ell}-\ell+1-\varkappa
\ge N_*+1=2N_4+1$. With \eqref{yinteriorest} this yields the estimate
 \begin{equation} \label{tricky}
\begin{aligned}
& \sup_{j=0,\ldots,K}\|y_{\ell+\nu+n_+-1+j}-\bar{x}_{n_++j-1}-\bar{x}_{n_-+j-K+1} \|\\
\le & \ C_3\left(\|y_{\ell+\nu+n_+-1}-\bar{x}_{n_+-1}\| +
\|y_{\ell+\nu+n_+-1+K}-\bar{x}_{n_- +1}\| +|\lambda|+\alpha_*^{-K/2} \right) \\
\le & \ C_3 \left( 2 \varepsilon +|\lambda|+ \alpha_*^{N_4}\right).
\end{aligned}
\end{equation}
Finally, we use this to estimate for $n=p+\ell$ and $p=n_+-1,\ldots,\tilde{p}$
\begin{equation*} \label{yfinis}
\begin{aligned}
T_n  & \le  \|y_{n+\nu}-\bar{x}_{n-\ell} -\bar{x}_{n-\tilde{\ell}}\|
+ \sum_{m\in I(s),m\neq \ell,\tilde{\ell}}\|\bar{x}_{n-m} \|.
\end{aligned}
\end{equation*}
The first term is handled by \eqref{tricky}. Further note that 
\begin{equation} \label{restfinis}
 \sum_{\tilde{\ell}<m\in I(s)}\|\bar{x}_{n-m}\| \
\le C_e \sum_{j\ge 1}\alpha_*^{j N_*} \le \frac{C_e \alpha_*^{N_*}}{1-\alpha_*}.
\end{equation}
and the same estimate holds for $\ell>m \in I(s)$.
Collecting estimates \eqref{estmiddle} to \eqref{restfinis} we arrive at
\eqref{serest}:
\begin{equation*}
 T_n \le (2 C_3+1)\varepsilon +C_3 |\lambda| + C_3 \alpha_*^{N_4}
+ \frac{2 C_e}{1-\alpha_*} \alpha_*^{2 N_4} 
\le  C_4 \e^{-\alpha N_4}+ C_3 |\lambda| + (2 C_3+1) \varepsilon.
\end{equation*}
Finally, we note that \eqref{serest} also holds in case $\tilde{\ell}$
is the smallest index in $I(s)$ or $\ell$ is 
the largest index in $I(s)$, respectively.
 Then one repeats the previous arguments
with the formal setting $\ell=-\infty$ resp. $\tilde{\ell}=\infty$
and uses the corresponding one-sided version of Lemma \ref{L1}. 

\noindent
{\bf Step 3:} (Construction and estimate of $\tau$ and $x_{\Z}$)\\
We want to find $x_{\Z}\in \ell^{\infty}(\R^k),\tau \in \ell^{\infty}(s)$
such that \eqref{5.4} holds with $g=0$. The second term of the operator
$G$ in \eqref{defop} vanishes provided we solve
\begin{equation} \label{tausolve}
\sum_{m\in I(s)}\langle \beta^{-\ell} u_{\Z},\beta^{-m}u_{\Z}\rangle \tau_m
= \langle \beta^{-\ell}u_{\Z},\beta^{\nu}y_{\Z} - p_{\Z}(s) \rangle,
\quad \ell \in I(s),
\end{equation} 
and set
\begin{equation} \label{xdef}
x_{\Z} =\beta^{\nu}y_{\Z} - p_{\Z}(s) - v_{\Z}(s,\tau).
\end{equation}
By Lemma \ref{L5.2} the linear system \eqref{tausolve} has a unique solution
$\tau\in \ell^{\infty}(s)$ which satisfies by \eqref{serest}, \eqref{exdec},
\eqref{const1} 
\begin{equation*} 
\begin{aligned}
\|\tau\|_{\infty} & \le 2 \|u_{\Z}\|_{\ell^1} \sup_{n\in \Z}\|y_{n+\nu}-\sum_{m\in I(s)}\bar{x}_{n-m} \|\\
& \le  2 C_e \frac{1+\alpha_*}{1-\alpha_*} \left(C_4 \e^{-\alpha N_4}+C_3 |\lambda| +
(2C_3+1)\varepsilon \right).
\end{aligned}
\end{equation*} 
Using \eqref{epsconst}, \eqref{Nconst}, \eqref{lambdaconst} and 
$C_e \frac{1+\alpha_*}{1-\alpha_*}\le C_5$ we end up with
$\|\tau\|_{\infty} \le \rho_{\tau}$.
Next we estimate $x_{\Z}$ from \eqref{xdef}. Using \eqref{exdec}
it is easy to show that
\begin{equation} \label{usumest}
\|\sum_{\ell \in I(s)}\beta^{-\ell}u_{\Z} \|_{\infty} \le C_e \frac{3-\alpha_*}{1-\alpha_*}= C_5.
\end{equation}
Therefore, when using \eqref{serest} again, we find
\begin{equation*}
\begin{aligned}
\|x_{\Z}\|_{\infty} & \le  \|\beta^{\nu}y_{\Z}-p_{\Z}(s)\|_{\infty} + 
C_5 \rho_{\tau} \\
& \le  C_4 \e^{-\alpha N_4}+C_3 |\lambda| +(2C_3+1)\varepsilon
+ C_5 \rho_{\tau}.
\end{aligned}
\end{equation*}
The estimates from \eqref{rhotaudef}, \eqref{epsconst}, \eqref{Nconst},
\eqref{lambdaconst} then guarantee $\|x_{\Z}\|_{\infty}\le \rho_{x}$.

Therefore, we know that $G_s(x_{\Z},0,\tau,\lambda)=0$ and the tuple
$(x_{\Z},0,\tau,\lambda)$ lies in the balls in which \eqref{5.4}
has a unique solution. By uniqueness we conclude
$g_s(\tau,\lambda)=0$  and  $x_{\Z}=x_{\Z,s}(\tau,\lambda)$ for 
$\lambda\in \Lambda$. Moreover, by the defining equations \eqref{tausolve}
and \eqref{xdef} we  obtain that equality \eqref{1.20} holds.

{\bf Step 4:} (Proof of Theorem \ref{Th1.1} (ii))\\
The radius $r_{\tau}$ will be taken such that
\begin{equation} \label{rtaucond}
18C_3(C_0 C_2 r_{\tau}^2 + C_5 r_{\tau}) \le \varepsilon.
\end{equation}
Let $\tau \in B_{r_{\tau}} \subset \ell^{\infty}(s)$, $\lambda \in \Lambda$
satisfy $g_{s}(\tau,\lambda)=0$ for some $s\in \Omega_N$ and let $x_{\Z,s}$
be given as in Theorem \ref{Th5.1}. Then clearly, the sequence
\begin{equation*} 
y_{\Z}= x_{\Z,s}(\tau,\lambda)+p_{\Z}(s)+v_{\Z}(s,\tau)
\end{equation*}
is an orbit of \eqref{a1}. It remains to show that $y_n\in U$ for
all $n \in \Z$.
Application of \eqref{Ginvest} in Theorem \ref{Th5.1} and of Lemma \ref{L5.3}
yields the estimate
\begin{equation} \label{xfinest}
\begin{aligned}
\|x_{\Z,s}(\tau,\lambda)\|_{\infty} & \le  C_0
\|G_s(x_{\Z,s}(\tau,\lambda),g_s(\tau,\lambda),\tau,\lambda) 
-G_s(0,0,\tau,\lambda)\|_\infty \\
&=  C_0 \|F(p_{\Z}(s)+v_{\Z}(s,\tau),\lambda) \|_{\infty} \\
& \le  C_0 C_2(|\lambda| +\|\tau\|_\infty^2 + \e^{-\alpha N/2}).
\end{aligned}
\end{equation} 
From \eqref{usumest} we find
\begin{equation} \label{taufinest}
\|v_{\Z}(s,\tau) \|_{\infty} \le C_5 r_{\tau}.
\end{equation}
We estimate  the distance of $y_{\Z}$ to the centers of the 
balls $B_j$ in $U$ by showing for $\nu=-n_--1$
\begin{equation} \label{intyfinest}
\|y_{p+\ell+\nu}-\bar{x}_{p}\| \le \frac{\varepsilon}{6  C_3}
\quad \text{for} \quad p=n_-+1,\ldots,n_+-1, \ell \in I(s).
\end{equation}
Note that the right-hand side is less equal $\varepsilon$ since we chose
$C_3\ge1$ in Lemma \ref{L1}.
Using \eqref{xfinest}, \eqref{taufinest} and \eqref{estmiddle} we
obtain for $\ell \in I(s)$
\begin{equation} \label{yfinest}
\begin{aligned}
\|y_{p+\ell+\nu}-\bar{x}_{p}\| & = \|x_{p+\ell,s}(\tau,\lambda) +
\sum_{m\in I(s), m\neq \ell}\bar{x}_{p+\ell-m} + v_{p+\ell}(s,\tau) \| \\
& \le  C_0 C_2(|\lambda| +\|\tau\|_\infty^2 + \e^{-\alpha N/2})+C_5 r_{\tau}
+ \frac{2 C_e}{1-\alpha_*} \alpha_*^{N_4} \\
& \le  C_0C_2 |\lambda| + (C_0C_2r_{\tau}+C_5) r_{\tau} + 
C_6 \alpha_*^{N_4}.
\end{aligned}
\end{equation}
Conditions \eqref{Nconst}, \eqref{lambdaconst} and \eqref{rtaucond}
guarantee that \eqref{intyfinest} is satisfied.

Next we consider two consecutive indices $\ell < \tilde{\ell}$ in $I(s)$
and $n=p+\ell$ for $p=n_+-1,\ldots,\tilde{p}=\tilde{\ell}-\ell+n_-+1$.
In  the first step we 
show that $y_{p+\ell+\nu}\in U_3$.  Using $p\ge n_+-1$ and 
$p+\ell-\tilde{\ell}\le n_- +1$  and \eqref{restfinis} 
we estimate similar to \eqref{yfinest} 
\begin{equation*}
\begin{aligned}
\|y_{p+\ell+\nu}\| &\le \|x_{p+\ell,s}(\tau,\lambda)\| + \|\bar{x}_{p}\|
+\|\bar{x}_{p+\ell- \tilde{\ell}}\| +
\sum_{\ell,\tilde{\ell} \neq m\in I(s)}\|\bar{x}_{p+\ell-m}\| +
 \|v_{p+\ell}(s,\tau) \| \\
& \le  C_0 C_2(|\lambda| +\|\tau\|_\infty^2 + \e^{-\alpha N/2})+C_5 r_{\tau}
+ \frac{2 C_e}{1-\alpha_*} \alpha_*^{N_4} +2 \varepsilon_0 \\
& \le  C_0C_2 |\lambda| + (C_0C_2r_{\tau}+C_5) r_{\tau} + 
C_6 \alpha_*^{N_4} + 2 \varepsilon_0\\
& \le  \frac \varepsilon {6C_3}+2\varepsilon_0 \le 3 \varepsilon_0.
\end{aligned}
\end{equation*}
Since $B_{3\varepsilon_0}\subset U_3$ this proves our assertion.
Now we can invoke Lemma \ref{L1} and find as in \eqref{tricky} 
 \begin{equation} \label{tricky2}
 \sup_{j=0,\ldots,K}\|y_{\ell+\nu+n_+-1+j}-\bar{x}_{n_++j-1}-\bar{x}_{n_-+j-K+1} \|
 \le  C_3 \left( \frac{\varepsilon}{3C_3} +|\lambda|+ \alpha_*^{N_4}\right).
\end{equation}
For $j=0,\ldots,N_4$ we have $n_-+j-K+1\le -N_4+n_-$ and hence
\[
\|\bar{x}_{n_-+j-K+1}\|\le C_e \e^{\alpha(n_-+N_4-K+1)} \le C_e \alpha_*^{N_4},
\]
while for $j=N_4+1,\ldots K$ we have $n_++j-1 \ge n_+ +N_4$ and hence
\[ 
\|\bar{x}_{n_++j-1}\| \le C_e \alpha_*^{n_+ +N_4} \le C_e\alpha_*^{N_4}.
\]
Combining this with \eqref{tricky2} we find 
\begin{equation*} 
\begin{aligned}
\|y_{p+\ell+\nu}-\bar{x}_{p}\| &\le  \frac{\varepsilon}{3} +C_3|\lambda|
+ (C_3+C_e)\alpha_*^{N_4}, \quad p=n_+-1,\ldots,n_+-1+N_4, \\
\|y_{p+\ell+\nu}-\bar{x}_{p-\tilde{\ell}+\ell}\| & \le 
\frac{\varepsilon}{3} +C_3|\lambda|+ (C_3+C_e)\alpha_*^{N_4},  
\quad  p=n_++N_4,\ldots,\tilde{p}.
\end{aligned}
\end{equation*}
We have arranged the constants in \eqref{Nconst} and \eqref{lambdaconst}
such that the right hand side is bounded by $\varepsilon$.

For the final step we note that we just have shown (cf.\ \eqref{defV0})
\[y_{p+\ell+\nu} \in 
\left\{
\begin{array}{rl}
V_\varkappa, & \text{ for } p = n_+-1,\\
V_1, & \text{ for } p = \tilde p =\tilde{\ell}-\ell+n_-+1 ,\\
V_0,& \text{ for }  p=n_+,\ldots,\tilde{\ell}-\ell+n_-    
\end{array}
\right.
\]
for two consecutive indices $\ell < \tilde{\ell}$ in $I(s)$.
On the other hand we know from \eqref{intyfinest} that
 $y_{p+\ell+\nu}\in B_{\varepsilon}(\bar{x}_p), p=n_-+1,\ldots,n_+-1$. 
Since $y_{\Z}$ is an 
$f(\cdot,\lambda)$ orbit we conclude by induction from
the definition \eqref{vjdef} 
\[ y_{p+\ell+\nu}\in V_{p-n_-}, \quad \text{for} \quad
p=n_+-1,\ldots, n_-+1.
\]
Therefore the sequence  $y_{\Z}$ lies in $ U$ which proves our assertion.
  
{\bf Step 5:} (Proof of Theorem \ref{Th1.3} (i))
The proof of \eqref{xshift},\eqref{gshift} is easily accomplished by noting the equivariance relations
\begin{equation*} 
\begin{aligned} 
I(\beta s) &=I(s)-1, \quad p_{\Z}(\beta s) = \beta p_{\Z}(s), \quad
 v_{\Z}(\beta s, \beta \tau) = \beta v_{\Z}(s,\tau),\\
w(\beta s, \beta g) & = \beta w(s,g), \quad F(\beta x_{\Z},\lambda) =
\beta F(x_{\Z},\lambda), \quad \text{and} \\
\langle \beta^{\ell}u_{\Z},x_{\Z}\rangle &= 
\langle\beta^{\ell+1}u_{\Z},\beta x_{\Z} \rangle \quad \text{for} \quad
\ell\in I(\beta s)=I(s)-1.
\end{aligned}
\end{equation*}
The assertion then follows by uniqueness from Theorem \ref{Th5.1} since 
neighborhoods are shift invariant as well.

The proof of Theorem \ref{Th1.3}(ii) will be deferred to the next section.

\section{Proof of Reduction Theorem}\label{S6}
\subsection{Nonlinear estimates}
According to \eqref{defop} the Frechet derivative of $G_s$ w.r.t.\ $x_{\Z},g$
is  given by 
\begin{equation} \label{Gderive}
\begin{aligned}
D_{(x,g)}G_s(x_{\Z},g,\tau,\lambda)(y_{\Z},h) =&
\begin{pmatrix} D_xF(p_{\Z}(s)+x_{\Z}+v_{\Z}(s,\tau),\lambda)y_{\Z} +w(s,h) \\
\langle \beta^{-\ell} u_{\Z}, y_{\Z}\rangle ,\; \ell\in I(s)
\end{pmatrix}, \\
D_xF(x_{\Z},\lambda)y_{\Z}= & (y_{n+1}-f_x(x_n,\lambda)y_{n})_{n\in \Z}.
\end{aligned}
\end{equation} 
The key step in the proof of Theorem \ref{Th5.1} will be a uniform bound
for the inverse of $D_{(x,g)}G_s(0,0,0,0)$.
\begin{lemma} \label{Linv}
There exist constants $\hat{C},\hat{N}>0$ such that for all
 $N\ge \hat{N}$, $s\in \Omega_N$ the operator
 $ D_{(x,g)}G_s(0,0,0,0)$ is invertible and satisfies 
\begin{equation*} 
\|y_{\Z}\|_{\infty}+\|h\|_{\infty} \le \hat{C} \| D_{(x,g)}G_s(0,0,0,0)(y_{\Z},h)
\|, \quad y_{\Z}\in \ell^{\infty}(\R^k), h\in \ell^{\infty}(s).
\end{equation*}
\end{lemma}
Before proving Lemma \ref{Linv} we finish the proof of Theorem \ref{Th5.1}.

\begin{proof}
Let $L_x$ and $ L_{\lambda}$ be Lipschitz constants of the Jacobian 
$f_x(x,\lambda)$ with respect to $x$ and $\lambda$ in a compact ball
that contains the homoclinic orbit in its interior.
Then formula \eqref{Gderive} and the bound \eqref{usumest} directly
lead to the Lipschitz estimate
\begin{equation} \label{lipDG}
\begin{aligned}
& \|D_{(x,g)}G_s(x^1_{\Z},g_1,\tau_1,\lambda_1)-D_{(x,g)}G_s(x^2_{\Z},g_2,\tau_2,\lambda_2)\|
  \\
\le & \; L_x(\|x_{\Z}^1-x_{\Z}^2\|_{\infty}+C_5 \|\tau_1-\tau_2\|_{\infty}) + 
L_{\lambda}|\lambda_1-\lambda_2|
\end{aligned}
\end{equation}
for all $x_{\Z}^1,x_{\Z}^2\in B_{\rho_{x}}$, $g_1,g_2\in \ell^{\infty}$,
$\tau_1,\tau_2 \in B_{\rho_{\tau}}$, and $\lambda_1,\lambda_2\in \Lambda_0$
with $\rho_{x},\rho_{\tau}$ taken sufficiently small.
From Lemma \ref{Linv} and Lemma \ref{banachlemma} we obtain that 
the operators $D_{(x,g)}G_s(0,0,\tau,\lambda)$ are invertible for 
$\tau\in B_{\rho_{\tau}}, \lambda\in B_{\rho_{\lambda}}$ provided we choose
\[ N\ge \hat{N} \quad \text{and} \quad 
L_xC_5\rho_{\tau}+L_{\lambda}\rho_{\lambda} \le \frac{1}{2 \hat{C}}.\]
Then we have
\begin{equation*} 
\|D_{(x,g)}G_s(0,0,\tau,\lambda)^{-1}\| \le 2 \hat{C},\quad 
\tau\in B_{\rho_{\tau}},\lambda\in B_{\rho_{\lambda}}.
\end{equation*}
Now we apply Theorem \ref{Va2} to every operator $F=G_s(\cdot,\cdot,\tau,\lambda)$ in the spaces $Y=Z=\ell^{\infty}(\R^k)\times\ell^{\infty}(s)$.
Setting $\sigma=\frac{1}{2\hat{C}}$, $y_0=0$ and taking $L_x\rho_x\le \frac{1}{4\hat{C}}$
we find that condition \eqref{invlip} is satisfied with 
$\kappa=\frac{1}{4\hat{C}}$ and $\delta=\rho_x$.
Finally, we obtain from Lemma \ref{L5.3} for all $N\ge N_2$,
\begin{equation*}
\begin{aligned}
\|G_s(0,0,\tau,\lambda)\| &=  \|F(p_{\Z}(s)+v_{\Z}(s,\tau),\lambda)\|_{\infty}\\
& \le  C_2(|\lambda|+ \|\tau\|_{\infty}^2 + \e^{- \alpha N/2}) \\
& \le  C_2 (\rho_{\lambda}+\rho_{\tau}^2 + \e^{-\alpha N/2}).
\end{aligned}
\end{equation*}
Now we select $N_0\ge\max(\hat{N},N_2)$ and $\rho_{\lambda},\rho_{\tau}$
such that $ C_2 (\rho_{\lambda}+\rho_{\tau}^2 + \e^{-\alpha N_0}) \le
\frac{\rho_x}{4 \hat{C}} $. 
Then we find  
$\|G_s(0,0,\tau,\lambda)\|\le
 (\sigma-\kappa)\delta = \frac{\rho_x}{4 \hat{C}},
$
for all $N\ge N_0$, i.e. condition \eqref{V2} is satisfied. 
An application of Theorem \ref{Va2} finishes the proof.
\end{proof}

\begin{remark} If instead of Theorem \ref{Va2} we use a Lipschitz inverse mapping
theorem with smooth parameters (cf.\ \cite[Appendix]{ir02}) then it is
easily seen that the solutions $g_s,x_{\Z,s}$ are smooth functions of
$\tau$ and $\lambda$.
\end{remark}

\begin{proof}
{\bf (Theorem \ref{Th1.3} (ii))} \\
With $c_{\lambda}, c_{x}$ from \eqref{cdef} define  $h=h(\tau,\lambda)\in \ell^{\infty}(s)$ by 
\begin{equation*} 
h_{\ell}= c_{\lambda}\lambda +c_x \tau_{\ell}^2, \quad \ell\in I(s).
\end{equation*}
The idea is to construct elements 
$y_{\Z}=y_{\Z}(\tau,\lambda)\in \ell^{\infty}(\R^k)$ and
$\gamma=\gamma(\tau,\lambda) \in \ell^{\infty}(s)$ such that the residual
\begin{equation*} 
G_s(p_{\Z}(s)+y_{\Z}+v_{\Z}(s,\tau), h+\gamma,\tau,\lambda)
\end{equation*} 
and $\gamma$ are of higher order than $\mathcal{O}(|\lambda|+\|\tau \|^2_{\infty})$.
Then the assertion follows from \eqref{Ginvest} by comparing them
to $x_{\Z,s}(\tau,\lambda),g_s(\tau,\lambda)$.
We find $y_{\Z},\gamma$ by Taylor expansion of $F$
(we abbreviate $F_x^0=F_{x}(p_{\Z}(s),0)$ etc.)
\begin{equation*}
\begin{aligned}
F&(p_{\Z}(s)+y_{\Z}+v_{\Z}(s,\tau),\lambda) = F^0+
F_x^0(y_{\Z}+v_{\Z}(s,\tau)) +F_{\lambda}^0 \lambda \\
+& \; \frac{1}{2}F_{xx}^0(y_{\Z}+v_{\Z}(s,\tau))^2 
+F_{x,\lambda}^0(y_{\Z}+v_{\Z}(s,\tau))\lambda + \frac{1}{2} 
F_{\lambda \lambda}^0\lambda^2\\
+ & \; \mathcal{O}((|\lambda|+\|\tau\|_{\infty}+\|y_{\Z}\|_{\infty})^3).
\end{aligned}
\end{equation*} 
From the estimates \eqref{pseudofest}, \eqref{pseudouest}
in the proof of Lemma \ref{L5.3} we have
 \[ \|F^{0}\|_{\infty} =\mathcal{O}(\e^{-\alpha N/2}),
\quad \|F_x^0 v_{\Z}(s,\tau)\| = 
\mathcal{O}(\e^{-\alpha N/2}\|\tau\|_{\infty}).
\]
In a similar way, using Lipschitz constants for $f_{\lambda}$ and $f_{xx}$ we
find
\begin{equation*}
\begin{aligned}
 (F_{\lambda}^0)_n= & - \sum_{\ell\in I(s)}f_{\lambda}(\bar{x}_{n-\ell},0)
+ \mathcal{O}(\e^{-\alpha N/2}) \\
(F_{xx}^0 (v_{\Z}(s,\tau))^2)_n =& - \sum_{\ell\in I(s)} \tau_{\ell}^2
f_{xx}(\bar{x}_{n-\ell},0)(u_{n-\ell})^2 +
\mathcal{O}(\e^{-\alpha N/2}\|\tau\|_{\infty}).
\end{aligned}
\end{equation*}
Therefore, Taylor expansion of $G_s$ yields
\begin{equation} \label{taylorg}
\begin{aligned}
G_s&(p_{\Z}(s)+y_{\Z}+v_{\Z}(s,\tau), h+\gamma,\tau,\lambda) =  
D_{(x,g)}G_s^0 (y_{\Z},\gamma) +(\varphi_{\Z},0) \\
+& \; \mathcal{O}(\e^{-\alpha N/2}+ \|\tau\|_{\infty} |\lambda|
+\|\tau\|_{\infty} \|y_{\Z}\|_{\infty}  + \lambda^2  + \|y_{\Z}\|_{\infty}^2+  \|\tau\|_{\infty}^3),
\end{aligned}
\end{equation}
where
\[ \varphi_n= 
\sum_{\ell\in I(s)}\lambda(-f_{\lambda}(\bar{x}_{n-\ell},0)+
c_{\lambda} w_{n-\ell}) + \tau_{\ell}^2 (-\frac{1}{2}f_{xx}(\bar{x}_{n-\ell},0)
(u_{n-\ell})^2 +c_x w_{n-\ell}).
\]
This suggests to define $(y_{\Z},\gamma)$ by
\begin{equation} \label{defcorrect}
D_{(x,g)}G_{s}^0 (y_{\Z}, \gamma)  = -(\varphi_{\Z},0)  .
\end{equation}
From this equation and Lemma \ref{Linv} we have the estimate
\begin{equation} \label{ygamest}
 \|y_{\Z}\|_{\infty}+ \|\gamma\|_{\infty} \le \tilde{C} (|\lambda|
+\|\tau\|_{\infty}^2).
\end{equation} 
Taking the inner product of the first coordinate in \eqref{defcorrect} with  
$\beta^{-\ell}w_{\Z},\ell\in I(s)$ and using \eqref{cdef}, \eqref{rangechar}
leads to the improved estimate
\begin{equation}\label{gammaimprove}
\|\gamma\|_{\infty} \le \tilde{C}\e^{-\alpha N/2} (|\lambda|+\|\tau\|_{\infty}^2).
\end{equation}
By \eqref{ygamest}, \eqref{defcorrect} the Taylor expansion
\eqref{taylorg} of $G_s$ assumes the form
\begin{equation}\label{estres}
\|G_s(p_{\Z}(s)+y_{\Z}+v_{\Z}(s,\tau), h+\gamma,\tau,\lambda)\| =
\mathcal{O}(\e^{-\alpha N/2}+ \|\tau\|_{\infty} |\lambda|
 +\lambda^2 +  \|\tau\|_{\infty}^3).
\end{equation}
With \eqref{Ginvest} and \eqref{gammaimprove} this leads us to
the final result
\[\|x_{\Z,s}(\tau,\lambda)-y_{\Z}(\tau,\lambda)\|_{\infty}
+\|g_s(\tau,\lambda)-h(\tau,\lambda)\|_{\infty} = 
\mathcal{O}(\e^{-\alpha N/2}+ \|\tau\|_{\infty} |\lambda|
 +\lambda^2 +  \|\tau\|_{\infty}^3).
\]
\end{proof}

\subsection{Linear estimates} \label{sublinear}
In this subsection we prove Lemma \ref{Linv}.
For any two integers $n_l \le n_r$ and for any number $a \ge 0$
consider the weight function
\begin{equation} \label{weight}
\omega_n = \omega_n(a,n_l,n_r) = \min\left(\e^{a(n-n_l)},1,\e^{a(n_r-n)}\right), \quad n\in \Z.
\end{equation}
Note that $\omega_{\Z}$ has a constant plateau of arbitrary width
 with exponentially decaying tails on both sides. 
We also allow $n_l=-\infty$ and $n_r=\infty$ (but neither $n_l=n_r=-\infty$ nor $n_l=n_r=\infty$), in which case $\omega_{\Z}$ has only one-sided decay
or degenerates to the maximum norm if $n_l=-\infty$ and $n_r=\infty$. 
In the following we will suppress the dependence of the norm on the data
$n_l,n_r,a$, but all our estimates will be uniform with respect to 
\begin{equation} \label{constrestrict}
-\infty \le n_l \le n_r \le \infty \quad 0\le a \le a_0 <\alpha,
\end{equation}
where $0<a_0<\alpha$ is fixed. The following lemma shows that
exponentially decaying kernels preserve the weight.
 \begin{lemma} \label{plateau}
There exists a constant $K_0$, depending only on $a_0,\alpha$, such that
\begin{equation*} 
\sum_{m\in \Z} \e^{-\alpha|n-m-1|}\omega_m \le K_0 \, \omega_n, \quad
\text{for all} \quad n \in \Z
\end{equation*}
and for all weight functions satisfying \eqref{constrestrict}.
\end{lemma}

\begin{proof} We consider
\[\delta_n = \sum_{m\in \Z}\omega_n^{-1}\e^{-\alpha|n-m-1|}\omega_m \]
only for $n\le n_l$ and leave cases $n_l+1\le n \le n_r$ and
$n_r+1 \le n$ to the reader.
\begin{equation*} 
\begin{aligned}
\delta_n = & \sum_{m \le n-1} \e^{a(n_l-n)-\alpha(n-1-m)+a(m-n_l)} 
+  \sum_{m=n}^{n_l-1}\e^{a(n_l-n)-\alpha(m+1-n)+a(m-n_l)} \\
+ & \sum_{m=n_l}^{n_r-1}\e^{a(n_l-n)-\alpha(m+1-n)} 
+  \sum_{m\ge n_r}\e^{a(n_l-n)-\alpha(m+1-n)+ a(n_r-m)}\\
=& \e^{-a} \sum_{m\le n-1}\e^{-(\alpha+a)(n-1-m)}+
\e^{-\alpha}\sum_{m=n}^{n_l-1}\e^{-(\alpha-a)(m-n)} \\
+ &  \e^{-(\alpha-a)(n_l-n)-\alpha} \sum_{m=n_l}^{n_r-1}\e^{-\alpha(m-n_l)}
+ \e^{-(\alpha-a)(n_l-n)+\alpha(n_l-n_r-1)}\sum_{m\ge n_r}\e^{-(\alpha +a)(m-n_r)}\\
\le & \frac{\e^{-a}+\e^{-\alpha}}{1 - \e^{-(\alpha+a)}} +
\e^{-\alpha}\left(\frac{1}{1-\e^{-(\alpha-a)}} + \frac{1}{1- \e^{-\alpha}} \right).
\end{aligned}
\end{equation*}
A suitable constant for all cases is $K_0 = 2\frac{1+\e^{-\alpha}}{1-\e^{-\alpha}}
+ \e^{a_0}+ \frac{\e^{a_0}}{1-\e^{-(\alpha-a_0)}}$.
\end{proof}

With the weights from above we consider the Banach space 
\begin{equation*} 
\ell_{\omega}^{\infty}= \{ y_{\Z}\in (\R^k)^{\Z}:
\|y_{\Z}\|_{\omega} =\sup_{n\in \Z}\|y_n \omega_n^{-1}\|< \infty \}.
\end{equation*} 
Taking the exponent $\alpha$ from \eqref{exdec}, \eqref{exdectrans} we have
 the following result for the variational equation \eqref{vareq}.
\begin{lemma} \label{lsing}
Let conditions {\bf A1 -  A3} and {\bf B4} hold. Then the linear system
\begin{equation} \label{vartaneq}
\begin{aligned}
y_{n+1}-f_x(\bar{x}_n,0)y_n + h w_n & =  r_n, \quad n \in \Z \\
\langle u_{\Z}, y_{\Z}\rangle & =  \gamma 
\end{aligned}
\end{equation}
has a unique solution $(y_{\Z},h)\in \ell_{\omega}^{\infty} \times \R$ for
every $(r_{\Z},\gamma)\in \ell_{\omega}^{\infty} \times \R$. Moreover there
is a constant $C^*$ such that for all weights \eqref{weight} with \eqref{constrestrict},
\begin{equation} \label{singgen}
\|y_{\Z}\|_{\omega}+|h| \le C^* (\|r_{\Z}\|_{\omega}+ |\gamma|\|u_{\Z}\|_{\omega}).
\end{equation}
If  $n_l\le 0\le n_r$, then \eqref{singgen} simplifies to
\begin{equation*}
\|y_{\Z}\|_{\omega} + |h| \le C^* \left( \|r_{\Z}\|_{\omega} +|\gamma|\right).
\end{equation*}
\end{lemma}

 \begin{proof} Let us abbreviate $A_n=f_{x}(\bar{x}_n,0)$,
$\mathcal{L}y_{\Z}= (y_{n+1}-A_n y_n)_{n\in \Z}$ and denote by
$\Phi$ the solution operator \eqref{defPhi} of \eqref{app4}.  
From \cite[Section 2]{pa88} (see also \cite{kl98a})
 we obtain the following facts. Equation \eqref{app4}
has an exponential dichotomy for $n\ge 0$ with data $(K,\alpha,P_n^{+s},P_n^{+u})$
and for $n\le 0$ with data $(K,\alpha,P_n^{-s},P_n^{-u})$ (see Definition 
\ref{b14}). Due to {\bf B4} we have $\range(P_0^{+s})\cap \range(P_0^{-u}) =
\mathrm{span}\{u_0\}=:Y_0$ and there exist decompositions
\begin{equation} \label{decompspace}
\range(P_0^{+s}) = Y_0 \oplus Y_+, \quad \range(P_0^{-u})=Y_0 \oplus Y_- ,
\quad Y_+ \cap Y_- = \{ 0\},
\end{equation}
where 
\begin{equation*}
\begin{aligned}
 k_s&=  \mathrm{rank}(P_n^{+s})= \dim(Y_+)+1,\quad (n\ge 0),\\
k_u & =  \mathrm{rank}(P_n^{-u})=\dim(Y_-)+1,\quad (n\le 0), \\
k_u & =  k-k_s, \quad P_n^{+s} + P_n^{+u} = I \,(n\ge 0), \;
P_n^{-s} + P_n^{-u} = I \,(n\le 0).
\end{aligned}
\end{equation*}
The operator $\mathcal{L}:\ell^{\infty}(\R^k)\rightarrow \ell^{\infty}(\R^k)$
is Fredholm of index $0$ with 
\begin{equation} \label{rangechar}
\kernel(\mathcal{L})=\Span\{u_{\Z}\}, \quad \range(\mathcal{L}) = \{ r_{\Z}\in \ell^{\infty}(\R^k): 
\langle w_{\Z},r_{\Z}\rangle=0 \}.
\end{equation}
One can also choose the ranges of $P_0^{+u}$ and
$P_0^{-s}$ such that, in addition to \eqref{decompspace},
\begin{equation} \label{modproj}
\R^k=Y_{0}\oplus Y_+ \oplus Y_- \oplus Y_1, \quad \range(P_0^{+u})=Y_- \oplus Y_1,
\quad \range(P_0^{-s})= Y_+ \oplus Y_1.
\end{equation}
Here $\dim Y_1=1$ and one can take $Y_1=\mathrm{span}\{w_{-1}\}$.
Following \cite[Lemma 2.7]{pa88} the general bounded solution
of $y_{n+1}-A_n y_n = r_n, n\ge 0 $ is given by 
\begin{equation} \label{yplusrep}
 y^+_{n} = \Phi(n,0) \eta_+ + \sum_{m\ge 0}G_+(n,m+1) r_m, \quad n\ge 0,
\quad \eta_+ \in \range(P_0^{+s})
\end{equation}
with the Green's function defined as follows
\begin{equation} \label{greenplus}
G_+(n,m) = \left\{
\begin{array}{cc}
\Phi(n,m) P_m^{+s}, & 0 \le m \le n, \\
-\Phi(n,m) P_m^{+u}, & 0\le n<m.
\end{array}
\right.
\end{equation}
Similarly, all bounded solutions of  $y_{n+1}-A_n y_n = r_n, n\le -1 $
are given by
\begin{equation} \label{yminusrep}
 y^-_{n} = \Phi(n,0) \eta_- + \sum_{m\le -1}G_-(n,m+1) r_m, \quad n\le 0,
\quad \eta_- \in \range(P_0^{-u}),
\end{equation}
where
\begin{equation} \label{greenminus}
G_-(n,m) = \left\{
\begin{array}{cc}
\Phi(n,m) P_m^{-s}, & m \le n \le 0, \\
-\Phi(n,m) P_m^{-u}, &  n<m \le 0.
\end{array}
\right.
\end{equation}
By the exponential dichotomies the Green's functions satisfy
\begin{equation} \label{greenest}
 \| G_{\pm}(n,m)\| \le K \e^{-\alpha|n-m|}, \quad n,m \in \Z_{\pm}.
\end{equation}
With $u_{\Z}^T y_{\Z}:= \langle u_{\Z},y_{\Z}\rangle$ we
 may write \eqref{vartaneq} in block operator form as
\begin{equation} \label{blockop}
\begin{pmatrix} \mathcal{L} & w_{\Z} \\
                u_{\Z}^T & 0 
\end{pmatrix}
\begin{pmatrix} y_{\Z} \\ h \end{pmatrix}
= \begin{pmatrix} r_{\Z} \\ \gamma \end{pmatrix}.
\end{equation}
By the bordering lemma \cite[Appendix]{be90} the block operator is Fredholm 
of the same index $0$ as $\mathcal{L}$ and, using \eqref{rangechar}, it
is a linear homeomorphism in $\ell^{\infty}(\R^k)\times \R$. Since $\ell_{\omega}^{\infty}$
is a closed subspace of  $\ell^{\infty}(\R^k)$ it suffices to prove that
the unique solution $(y_{\Z},h)$  of \eqref{blockop} in $\ell^{\infty}(\R^k)\times \R$  satisfies the estimate \eqref{singgen}
in case $r_{\Z} \in  \ell_{\omega}^{\infty}$.

Take the inner product of the first equation of \eqref{vartaneq} with $w_{\Z}$.
Then \eqref{rangechar} and the normalization $\|w_{\Z}\|_{\ell^2}=1$
show $h = \langle w_{\Z},r_{\Z}\rangle$. Therefore, by \eqref{exdectrans}, 
\begin{equation} \label{tauomegaest}
 |h| \le C_e \|r_{\Z}\|_{\omega} \sum_{n\in \Z}\e^{-\alpha |n|}\omega_n
\le C_e \frac{1+\e^{-\alpha}}{1-\e^{-\alpha}} \|r_{\Z}\|_{\omega} .
\end{equation}
With this $h$ we have $z_{\Z}:=r_{\Z}-h w_{\Z} \in \range(\mathcal{L})$ by
\eqref{rangechar}. Below we will construct a special solution 
$\hat{y}_{\Z}\in \ell^{\infty}(\R^k)$ of $\mathcal{L} \hat{y}_{\Z}=z_{\Z}$ such that
for some constant $C>0$
\begin{equation} \label{solspecial}
\|\hat{y}_{\Z} \|_{\omega} \le C \|z_{\Z} \|_{\omega}.
\end{equation}
By \eqref{rangechar} and \eqref{unormalize} the solution of \eqref{blockop} 
is given by
\begin{equation} \label{solformula}
(y_{\Z},h)= (\hat{y}_{\Z}+c u_{\Z},h), \quad c=\gamma-\langle u_{\Z},\hat{y}_{\Z}\rangle, 
\quad h = \langle w_{\Z},r_{\Z}\rangle.
\end{equation}
  From the exponential decay \eqref{exdec}, \eqref{exdectrans}
and Lemma \ref{plateau}, applied to the kernel 
$\e^{-\alpha(|n|+|m|)} \le \e^{\alpha}\e^{-\alpha|n-m-1|}$, we obtain
with $C'=C_e \e^{\alpha}K_0$,
\[ \|\langle u_{\Z},\hat{y}_{\Z}\rangle u_{\Z}\|_{\omega} \le
C' \|\hat{y}_{\Z}\|_{\omega}, \quad
\|\langle w_{\Z},r_{\Z}\rangle w_{\Z}\|_{\omega} \le
C' \|r_{\Z}\|_{\omega}.
\]
Then \eqref{tauomegaest}-\eqref{solformula} yield the assertion
\begin{equation*}
\begin{aligned}
\|y_{\Z}\|_{\omega} \le  (1+C')\|\hat{y}_{\Z}\|_{\omega}+ |\gamma| \|u_{\Z}\|_{\omega} \le  C(1+C')^2\|r_{\Z}\|_{\omega} + |\gamma| \|u_{\Z}\|_{\omega}.
\end{aligned}
\end{equation*}
It remains to construct $\hat{y}_{\Z}$ with $\mathcal{L}\hat{y}_{\Z}=z_{\Z}$ and
\eqref{solspecial}. We determine $\eta_+ \in \range(P^{+s}_0)$ and
$\eta_- \in \range(P^{-u}_0)$ such that $\hat{y}_n=y^+_n, n\ge 0$ with $y^+_n$ from
\eqref{yplusrep}, and $\hat{y}_n=y^-_n, n\le 0$ with $y^-_n$ from \eqref{yminusrep},
and such that the definitions coincide at $n=0$. The last condition
holds if and only if
\begin{equation} \label{jumpeta}
 \eta_+ - \eta_- = \sum_{m\le-1}G_-(0,m+1)z_m -\sum_{m\ge 0}G_+(0,m+1)z_m
=: \Delta_0.
\end{equation}
By \eqref{greenplus}, \eqref{greenminus}, \eqref{modproj} the first sum on the 
right is in 
$Y_+\oplus Y_1$ and the second sum is in $Y_-\oplus Y_1$ while the left-hand
side is in $Y_0\oplus Y_-\oplus Y_+$. Since $z_{\Z}\in \range(\mathcal{L})$ holds,
equation \eqref{jumpeta} has a solution and thus $\Delta_0 \in Y_+\oplus Y_-$.
We conclude from \eqref{modproj} that $\eta_+=P^{+s}_0 \Delta_0$ and
$\eta_-=- P^{-u}_0 \Delta_0$ are the unique solutions of \eqref{jumpeta}.
 With \eqref{greenest} and Lemma \ref{plateau} we estimate for $n\ge 0$
\begin{equation*}
\|\hat{y}_n\| \omega_n^{-1} \le  K \|z_{\Z}\|_{\omega}
\sum_{m\ge 0}\omega_n^{-1} \e^{-\alpha|n-m-1|}\omega_m 
\le   K K_0 \|z_{\Z}\|_{\omega}.
\end{equation*}
 An analogous estimate holds for $n \le 0$ and this completes the proof.
\end{proof}

\begin{proof} {\bf (Lemma \ref{Linv})}
We use Lemma \ref{banachlemma} and construct an approximate right inverse
$B_+$ of $D_{(x,g)}G_s^0=D_{(x,g)}G_s(0,0,0,0)$. For any $\ell\in I(s)$ we
define the interval
$J(\ell) =  \{\ell_-,\ldots,\ell,\ldots,\ell_+ \}$ where right and left
 neighbors are given by
\begin{equation} \label{lplusdef}
\ell_+  =  \left\{ \begin{array}{cc}
          \infty, & \text{if}\quad \ell=\max I(s),\\ 
                    \ell+ N_*, & \text{otherwise},
                    \end{array}
            \right. 
\end{equation}
\begin{equation}\label{lminusdef}
\ell_-  =  \left\{ \begin{array}{cc}
           
          -\infty, & \text{if}\quad \ell=\min I(s),\\ 
                    \max\{\tilde{\ell}_+: \tilde{\ell} \in I(s),
 \tilde{\ell}<\ell\}+1, & \text{otherwise}.
                    \end{array}
            \right. 
\end{equation}
Note that the sets $J(\ell),\ell\in I(s)$ define a partitioning of $\Z$.
In the following we consider $N\ge \hat{N}=2 N_*+1$ which implies
$\ell-\ell_- \ge \ell_+-\ell \ge N_*$. During the proof $N_*$ will be taken
sufficiently large. 
Given an element
$(r_{\Z},\gamma)\in \ell^{\infty}(\R^k)\times \ell^{\infty}(s)$, we decompose
\[r_{\Z}= \sum_{\ell \in I(s)} \one_{J(\ell)} r_{\Z}, \quad
\text{where} \quad (\one_{J(\ell)})_n =\left\{
\begin{array}{cc} 1, & n\in J(\ell), \\ 0, & \text{else} \end{array}
\right.
\]
 is the characteristic function of $J(\ell)$.
Let $B_0$ denote the solution operator of \eqref{blockop}, then we set
\begin{equation} \label{solvelocal}
(y^{\ell}_{\Z},h_{\ell}) = B_0(\beta^{\ell}\one_{J(\ell)}r_{\Z}, \gamma_{\ell})
\quad \text{for} \quad \ell\in I(s),
\end{equation}
and define $B_+$ as a blockwise inverse via
\begin{equation} \label{bplusdefine}
B_+(r_{\Z},\gamma)= (y_{\Z},h)= (\sum_{\ell\in I(s)}\beta^{-\ell}y^{\ell}_{\Z},
(h_{\ell})_{\ell \in I(s)}).
\end{equation}
Using Lemma \ref{lsing} with the settings $n_l=\ell_- -\ell$, 
$n_r=\ell_+-\ell$ we obtain a bound 
 \begin{equation} \label{ylbound}
\|y_{\Z}^{\ell}\|_{\omega}+ |h_{\ell}| \le C^*\left(\|\beta^{\ell} 
\one_{J(\ell)}r_{\Z}\|_{\omega}+ |\gamma_{\ell}|\right) \le
C^*\left(\|r_{\Z}\|_{\infty}+ |\gamma_{\ell}| \right).
\end{equation}
Let us abbreviate the weights from \eqref{weight},
\begin{equation*}
\omega_{n,\ell}=\omega_{n}(a,\ell_-,\ell_+) , \quad n\in \Z, \ell\in I(s).
\end{equation*} 
 Then equation \eqref{bplusdefine} and \eqref{ylbound} lead to the estimate
\begin{equation*} 
\begin{aligned}
\|y_n\| \le  \|\sum_{\ell\in I(s)} y^{\ell}_{n-\ell} \| 
\le  C^* \left(\|r_{\Z}\|_{\infty}+\|\gamma\|_{\infty}\right) 
\sum_{\ell \in I(s)} \omega_{n,\ell}, 
\end{aligned}
\end{equation*}

For $2\e^{-aN^*}\le 1$  the last sum is
 bounded by $1+4 \e^{-a}$ and \eqref{ylbound} yields
\begin{equation} \label{bplusest}
\|B_+(r_{\Z},\gamma)\|= \|y_{\Z}\|_{\infty}+\|h\|_{\infty}
\le (2C^*+1+4 \e^{-a})(\|r_{\Z}\|_{\infty}+\|\gamma\|_{\infty}).
\end{equation}
In the next step we show for $N^*$ sufficiently large,
\begin{equation} \label{contractright}
\|(r_{\Z},\gamma)-D_{(x,g)}G_s^0B_+(r_{\Z},\gamma)\| \le \frac{1}{2}
(\|r_{\Z}\|_{\infty}+\|\gamma\|_{\infty}),
\end{equation}
which by \eqref{estinvA}, \eqref{bplusest} gives the desired estimate 
\[ \|(D_{(x,g)}G_s^0)^{-1}\| \le 2(2C^*+1+4 \e^{-a}).
\]
We introduce $(\tilde{r}_{\Z},\tilde{\gamma})=
(r_{\Z},\gamma)-D_{(x,g)}G_s^0(y_{\Z},h)$.
From \eqref{bplusdefine} and the variational equation in \eqref{solvelocal} we have
\begin{equation*}
\begin{aligned}
\tilde{r}_n =& r_n - \Big(\sum_{m\in I(s)}y_{n+1-m}^m -f_x(\sum_{p\in I(s)}\bar{x}_{n-p},0)y_n+\sum_{m\in I(s)}h_mw_{n-m} \Big)\\
= & \sum_{m\in I(s)}\Big[ f_x(\sum_{p\in I(s)}\bar{x}_{n-p},0)
-f_x(\bar{x}_{n-m},0) \Big] y_{n-m}^m.
\end{aligned}
\end{equation*}
For $n\in \Z$ there exists a unique $\ell\in I(s)$ such that $n\in J(\ell)$.
Define the neighborhood $\mathcal{U}(\ell)$ of $\ell$ by 
$\mathcal{U}(\ell)=\{\hat{\ell},\ell,\tilde{\ell} \}$ with left neighbor $\hat{\ell}= \sup\{m\in I(s):m<\ell\}$ and 
right neighbor $\tilde{\ell}= \inf\{m\in I(s):m>\ell\}$ (as usual let
$\hat{\ell}=-\infty$ and $\tilde{\ell}=\infty$ if the sets are empty).
Using the Lipschitz constant $L_x$ of $f_x$ 
and  \eqref{pseudoest}, \eqref{ylbound} we obtain
\begin{equation*} 
\begin{aligned}
\|\tilde{r}_n \| &\le  L_x \Big( \sum_{m\in I(s)\setminus \mathcal{U}(\ell)}
 \|y_{n-m}^m\| \sum_{m\neq p\in I(s)} \|\bar{x}_{n-p}\|
 +\sum_{m\in\mathcal{U}(\ell)}
\|y_{n-m}^m\| \sum_{m\neq p\in I(s)}\| \bar{x}_{n-p}\| \Big) \\
& \le  L_{x}C^*(\|r_{\Z}\|_{\infty}+\|\gamma\|_{\infty})
\Big\{ \bar{C} \sum_{m\in I(s)\setminus \mathcal{U}(\ell)}\omega_{n,m}
+ \sum_{\ell\neq p \in I(s)}\|\bar{x}_{n-p}\| \Big. \\ 
& +  \Big. \omega_{n,\tilde{\ell}}\Big(\|\bar{x}_{n-\ell}\|
+\sum_{p\in I(s)\setminus\{\ell,\tilde{\ell}\}}\|\bar{x}_{n-p}\|\Big)
+\omega_{n,\hat{\ell}}\Big(\|\bar{x}_{n-\ell}\|
+\sum_{p\in I(s)\setminus\{\ell,\hat{\ell}\}}\|\bar{x}_{n-p}\|\Big)
\Big\}. 
\end{aligned}
\end{equation*}
We show that the terms in $\{ \ldots\}$ are of order 
$\mathcal{O}(\e^{-a N*})$ so that the contraction estimate \eqref{contractright} holds for the first component if $N^*$ is sufficiently large.
A critical term on the right-hand side is
\begin{equation}\label{critterm}
\omega_{n,\tilde{\ell}}\|\bar{x}_{n-\ell}\| \le  C_e 
\e^{-\alpha|n-\ell|+a(n-\tilde{\ell}_-) }
\le  C_e \e^{a(\ell -\tilde{\ell}_-)} \le C_e \e^{-a N^*}, \quad
\ell_- \le n \le \ell_+.
\end{equation} 
The term $\omega_{n,\hat{\ell}}\|\bar{x}_{n-\ell}\|$ is handled 
analogously. Further, 
\begin{equation*} \label{homsumest}
\begin{aligned}
\sum_{\ell\neq p \in I(s)}\|\bar{x}_{n-p}\| & \le 
C_e\Big( \sum_{\ell>p\in I(s)}\e^{-\alpha(n-p)} + \sum_{\ell<p\in I(s)}
\e^{-\alpha(p-n)} \Big) \\
& \le  C_e\Big( \sum_{\ell>p\in I(s)}\e^{-\alpha(\ell_--p)} + 
\sum_{\ell<p\in I(s)}
\e^{-\alpha(p-\ell_+)} \Big) \le 2 C_e \frac{\e^{-\alpha N_*}}{1- \e^{-\alpha}}.
\end{aligned}
\end{equation*}
The remaining terms allow similar estimates since $n$ always lies in 
an exponential decaying tail of the weights and of the shifted homoclinic orbits. 
Finally, we use \eqref{exdec} and \eqref{solvelocal}, \eqref{ylbound} for $\ell \in I(s)$, 
\begin{equation*} 
\begin{aligned}
|\gamma_{\ell}- \langle \beta^{-\ell}u_{\Z},y_{\Z} \rangle| &= 
|\sum_{n\in \Z}\langle u_{n-\ell},\sum_{\ell \neq m \in I(s)} y_{n-m}^m 
\rangle| \\
& \le 
C_e C^*(\|r_{\Z}\|_{\infty}+\|\gamma\|_{\infty}) \sum_{n\in \Z} \e^{-\alpha|n-\ell|}
\sum_{\ell\neq m\in I(s)}\omega_{n,m}.
\end{aligned}
\end{equation*}
We estimate the remaining sum by using \eqref{critterm}
\begin{equation*} 
\begin{aligned}
\sum_{n\in \Z} \e^{-\alpha|n-\ell|}
\sum_{\ell\neq m\in I(s)}\omega_{n,m}  & \le  
\sum_{n=\ell_-}^{\ell_+}\e^{-\alpha|n-\ell|}\Big(\omega_{n,\tilde{\ell}}+
\omega_{n,\hat{\ell}} + \sum_{m\in I(s)\setminus \mathcal{U}(\ell)} \omega_{n,m}
\Big) \\
& \quad \; +  C \Big(\sum_{n=\ell_+}^{\infty}\e^{-\alpha(n-\ell)} + 
\sum_{n=-\infty}^{\ell_- - 1}\e^{-\alpha(\ell -n)}\Big). 
\end{aligned}
\end{equation*}
The last two terms are $\mathcal{O}(\e^{-\alpha N^*})$ since $\ell_+-\ell \ge N^*$
and $\ell -\ell_- \ge N^*$.
Furthermore, we have for $\ell_- \le n \le \ell_+$ 
\begin{equation*}
\begin{aligned}
 \sum_{m\in I(s)\setminus \mathcal{U}(\ell)} \omega_{n,m} \le &
\sum_{\tilde{\ell}<m\in I(s)} \e^{a(n-m_-)} + \sum_{\hat{\ell}>m \in I(s)}
\e^{a(m_+-n)} \\
\le &\sum_{\tilde{\ell}<m\in I(s)} \e^{a(\ell_- -m_-)} + \sum_{\hat{\ell}>m \in I(s)}
\e^{a(m_+-\ell_+)} \le \frac{2 \e^{-2a N^*}}{1-\e^{-a N^*}}.
\end{aligned}
\end{equation*}
Our final estimate is
\begin{equation*} 
\begin{aligned}
\sum_{n=\ell_-}^{\ell_+} \e^{-\alpha|n-\ell|}\omega_{n,\tilde{\ell}}
 & \le  \sum_{n=\ell_-}^{\ell} \e^{-\alpha(\ell-n)+a(n-\tilde{\ell}_-)} 
+ \sum_{n=\ell+1}^{\ell_+} \e^{-\alpha(n-\ell)+a(n-\tilde{\ell}_-)} \\
& \le  \frac{\e^{-a N^*}}{1-\e^{-\alpha}} + \frac{\e^{-a N^*}}{1-\e^{-(\alpha-a)}}.
\end{aligned}
\end{equation*} 
The term $\sum_{n=\ell_-}^{\ell_+} \e^{-\alpha|n-\ell|}\omega_{n,\hat{\ell}}$
is estimated in a similar way. This finishes the proof of 
\eqref{contractright}.

We take the same $B_-=B_+$ as an approximate left inverse in Lemma
\ref{banachlemma}. It is convenient to require
$0<2a\le a_0<\alpha$ instead of \eqref{constrestrict}.
  Given $(y_{\Z},h)\in \ell^{\infty}(\R^k)\times 
\ell^{\infty}(s)$, we  set
\begin{equation*} 
(\eta_{\Z},t)= B_+ D_{(x,g)}G_s^0 (y_{\Z},h),
\end{equation*}
then it is sufficient to show
\begin{equation} \label{smallcontract}
\|\eta_{\Z}-y_{\Z}\|_{\infty} + \|t-h\|_{\infty} \le 
\mathcal{O}(\e^{-(\alpha-a)N_*})(\|y_{\Z}\|_{\infty} + \|h \|_{\infty}).
\end{equation}
For a fixed $\ell \in I(s)$ we consider first the case
\begin{equation} \label{case1}
y_n= 0 \quad (n\le \ell_- \;\text{and}\; n> \ell_+),
\quad  h_m=0 \quad (m \neq \ell),
\end{equation}
and prove the stronger estimate (recall $\hat{N}=2 N_* +1$)
\begin{equation} \label{stronglocal}
\|\eta_{\Z}-y_{\Z}\|_{\omega} +|t_{\ell}-h_{\ell}|+
 \sup_{\ell\neq m\in I(s)}|t_m|
\e^{\alpha(|m-\ell|-\hat{N})} \le 
\mathcal{O}(\e^{-\alpha N_*})(\|y_{\Z}\|_{\omega} + |h_{\ell}|),
\end{equation}
where $\omega_n=\omega_n(a,\ell_-,\ell_+)$, see \eqref{weight}.
Then we consider the case
\begin{equation} \label{case2}
y_n=0 \quad (n\ne \ell_-), \quad h = 0
\end{equation}
and prove the estimate
\begin{equation} \label{strongpoint}
\|\eta_{\Z}-y_{\Z}\|_{\omega} + \sup_{m\in \Z}|t_m |
\e^{a|m-\ell_-|} \le 
\mathcal{O}(\e^{-aN_*})\|y_{\Z}\|_{\omega},
\end{equation}
where $\omega_n=\omega(a,\ell_-,\ell_-)$.
Let us first show that the general case \eqref{smallcontract} follows from \eqref{stronglocal} and \eqref{strongpoint}.
With $J^{0}(\ell)=\{\ell_-+1,\ldots,\ell_+\}$, $J^1(\ell)=\{\ell_-\}$ we 
decompose
\[ (y_{\Z},h) = \sum_{m\in I(s)}\left[ (\one_{J^0(m)}y_{\Z}, h_m \one_{\{m\}}) + (\one_{J^1(m)}y_{\Z},0) \right]
\]
and define
\[ (\eta_{\Z},t) = \sum_{m\in I(s)}B_+D_{(x,g)}G^0_s \left[ (\one_{J^0(m)}y_{\Z}, h_m \one_{\{m\}}) + (\one_{J^1(m)}y_{\Z},0) \right].
\]
Similar to the proof of \eqref{contractright} we combine the local estimates
\eqref{stronglocal} and \eqref{strongpoint},
\begin{equation*} 
\begin{aligned}
\|y_n-\eta_n\| &=  \mathcal{O}(\e^{-aN_*})
 \sum_{m\in I(s)} \left[\omega_n(a,m_-,m_+)(\|\one_{J^0(m)}y_{\Z}\|_{\omega}
+ |h_m|) \right. \\
 & \qquad \qquad \qquad \qquad + \left. \omega_n(a,m_-,m_-) \|\one_{\{m_-\}}y_{\Z}\|_{\omega}\right] \\
& =  \mathcal{O}(\e^{-aN_*})(\|y_{\Z}\|_{\infty}+\|h\|_{\infty})
 \sum_{m\in I(s)} \left[\omega_n(a,m_-,m_+)+ \omega_n(a,m_-,m_-)\right].\\
& = \mathcal{O}(\e^{-aN_*})(\|y_{\Z}\|_{\infty}+\|h\|_{\infty}).
\end{aligned}
\end{equation*}
Using the exponential weights of $t_m$ in \eqref{stronglocal} and \eqref{strongpoint} leads to the same estimate for $\|h-t\|_{\infty}$
and \eqref{smallcontract} is proved.

For the proof of \eqref{stronglocal} let 
$\gamma_m=\langle \beta^{-m}u_{\Z},y_{\Z}\rangle$ and note that \eqref{case1} implies $r_n=0$  for $n \notin J(\ell)$. Hence
$\eta_{\Z}^{\ell}=\beta^{\ell}y_{\Z}$ satisfies 
\begin{equation*} \eta_{n+1}^{\ell}-f_x(\bar{x}_n,0)\eta_n^{\ell} + h_{\ell}w_n
=r_{n+\ell} + \left[ f_x(p_{n+\ell}(s),0)-f_x(\bar{x}_n,0)\right]\eta_n^{\ell}.
\end{equation*}
Using Lemma \ref{lsing} and a Lipschitz estimate for $f_x$ we can compare with 
the solution 
$(y^{\ell}_{\Z},t_{\ell})=B_0 (\beta^{\ell}\one_{J(\ell)}r_{\Z},\gamma_{\ell})$ of \eqref{solvelocal}. This leads to
\[
\|\beta^{-\ell}y^{\ell}_{\Z}-y_{\Z}\|_{\omega} +|t_{\ell}-h_{\ell}|
= \mathcal{O}(\e^{-\alpha N^*}) (\|y_{\Z}\|_{\omega}+|h_{\ell}|).
\]
For the solutions $(y^{m}_{\Z},t_{m})=B_0 (\beta^{m}\one_{J(m)}r_{\Z},\gamma_{m})=B_0(0,\gamma_m)$ with $m\neq \ell$
Lemma \ref{lsing} gives
\[\|\beta^{-m}y^{m}_{\Z}\|_{\omega}+ |t_m| \le C^* |\gamma_m|
\le C_e \|y_{\Z}\|_{\omega} \sum_{n=\ell_-+1}^{\ell_+}\e^{-\alpha |n-m|}
= \mathcal{O}(\e^{-\alpha(|m-\ell|-N^*)})\|y_{\Z}\|_{\omega}.
\] 
Collecting the last two estimates we find that 
$(\eta_{\Z},t)=(\sum_{m\in I(s)}\beta^{-m}y^m_{\Z},(t_m)_{m\in I(s)})$
(cf.\ \eqref{bplusdefine}) satisfies inequality \eqref{stronglocal}.

In case condition \eqref{case2} holds, let 
\[ r_n = (D_xF(p_{\Z}(s),0)y_{\Z})_n \; (n\in \Z), \quad
\gamma_m= \langle \beta^{-m}u_{\Z},y_{\Z}\rangle \; (m\in I(s)),
\]
and note $r_n=0$ for $n\neq\ell_-,\ell_--1$ as well as
\begin{equation} \label{rightsparse}
r_{\ell_-} = - f_x(p_{\ell_-}(s),0)y_{\ell_-}, \quad
r_{\ell_- -1}= y_{\ell_-}.
\end{equation}
Moreover,
\begin{equation} \label{gammamest}
 |\gamma_m|=|u_{\ell_--m}^T y_{\ell_-}| \le C_e \e^{-\alpha|\ell_- - m|}
\|y_{\ell_-}\|, \quad m\in I(s).
\end{equation}
As in \eqref{solvelocal} let 
$(y_{\Z}^m,t_m)= B_0 (\beta^m\one_{J(m)}r_{\Z},\gamma_m)$ for $m\in I(s)$.
Recall $\ell_--1=\hat{\ell}_+$ from \eqref{lminusdef} where $\hat{\ell}\in I(s)$ is the left neighbor of $\ell$.
For $m\neq \ell,\hat{\ell}$ we have $y_{\Z}^m=\beta^{-m}\gamma_mu_{\Z}$, $t_m=0$
 by \eqref{solformula} and therefore by \eqref{gammamest},
\begin{equation} \label{ymrest}
\|y_{\Z}^m\|_{\omega} \le C_e^2 \|y_{\ell_-}\| \e^{-\alpha|\ell_--m|}
\sup_{n\in \Z} \e^{-\alpha|n-m|+a|n-\ell_-|} =
\mathcal{O}(\e^{-(\alpha-a)|\ell_--m|})\|y_{\ell_-}\|.
\end{equation}
We introduce the weights $\omega_n^*= \omega_n(2a,\ell_--\ell,\ell_--\ell)$.
 For $m=\ell$ we use \eqref{gammamest} and Lemma \ref{lsing} with $\omega^*$,
\begin{equation} \label{yomegaest}
\begin{aligned} 
\|y_{\Z}^{\ell}\|_{\omega^*} \le& 
  C^* \left( \|\beta^{\ell}\one_{J(\ell)}r_{\Z}\|_{\omega^*}
+|\gamma_{\ell}| \|u_{\Z}\|_{\omega^*}\right) \\
\le & C^*\left(1+ C \e^{-(\alpha-2a)|\ell_--\ell|} \right) \|y_{\ell_-}\|
\le C \|y_{\ell_-}\|.
\end{aligned}
\end{equation}
In a similar manner,
\begin{equation}\label{yhatest}
 \|y^{\hat{\ell}}_{\Z}\|_{\hat{\omega}^*}  \le C \|y_{\ell_-}\|
\end{equation}
holds for the weights 
$\hat{\omega}^*_n = \omega_n(2a,\hat{\ell}_+-\hat{\ell},\hat{\ell}_+-\hat{\ell})$.
The $t$-values satisfy
\begin{equation} \label{thatest} 
\begin{aligned}
|t_{\ell}| &= |u_{\ell_- -\ell}^T f_{x}(p_{\ell_-}(s),0)y_{\ell_-}|
             \le C \e^{-\alpha|\ell_- - \ell|}\|y_{\ell_-}\|, \\
|t_{\hat{\ell}}| & = |u_{\hat{\ell} -\hat{\ell}_+}^T y_{\ell_-}|
             \le C \e^{-\alpha|\hat{\ell} - \hat{\ell}_+|}\|y_{\ell_-}\|.
\end{aligned}
\end{equation}
In particular, this proves the $t$-estimates in \eqref{strongpoint}.
Next we estimate the difference 
$d_{\Z}= \beta^{-\ell}y_{\Z}^{\ell}+ \beta^{-\hat{\ell}}y_{\Z}^{\hat{\ell}}-y_{\Z}$
by using the exponential dichotomy on $\Z$ for the constant coefficient 
operator $\mathcal{L}_0 y_{\Z}=(y_{n+1}-A y_n)_{n\in \Z}$, $A=f_x(0,0)$.
From \eqref{rightsparse} and the definition of $y_{\Z}^{\ell},y_{\Z}^{\hat{\ell}}$ we find
\begin{equation*} 
\begin{aligned} 
d_{n+1}-A d_n = & \;  (f_{x}(\bar{x}_{n-\ell},0)-A)y_{n-\ell}^{\ell}
+ (f_x(\bar{x}_{n-\hat{\ell}},0)-A)y_{n-\hat{\ell}}^{\hat{\ell}} \\
 & - t_{\ell}w_{n-\ell} - t_{\hat{\ell}}w_{n-\hat{\ell}}
+(A-f_x(p_{\ell_-}(s),0))\delta_{n,\ell_-}y_{\ell_-} = \sum_{i=1}^5 T_i.
\end{aligned}
\end{equation*}
For every term  we show $\|T_i\| =\mathcal{O}( \e^{-aN_*})\omega_n
 \|y_{\ell_-}\|$ with weights  
$\omega_n=\e^{-a |n-\ell_-|}$. In fact, our previous 
estimates \eqref{yomegaest} - \eqref{thatest} yield
\begin{equation*} 
\begin{aligned}
\|T_1\|\omega_n^{-1} & \le C \|\bar{x}_{n-\ell}\|\e^{-a|n-\ell_-|}
\|y_{\Z}^{\ell}\|_{\omega^*} \\
& \le C \|y_{\ell_-}\|\e^{-\alpha|n-\ell|-a|n-\ell_-|} \le 
C \e^{-a|\ell-\ell_-|}\|y_{\ell_-}\|  =\mathcal{O}( \e^{-aN_*})
 \|y_{\ell_-}\|, \\
\|T_2\|\omega_n^{-1} & \le C\e^{-\alpha|n-\hat{\ell}|+ a |n-\ell_-|
-2a|n-\hat{\ell}_+|} \|y_{\ell_-}\| \le C \e^{-a|\hat{\ell}-\hat{\ell}_+|}
\|y_{\ell_-}\|=\mathcal{O}( \e^{-aN_*}) \|y_{\ell_-}\|, \\
\|T_3\|\omega_n^{-1} & \le C|t_{\ell}|\e^{a|n-\ell_-|-\alpha|n-\ell|}\le
C\e^{-(\alpha-a)|\ell-\ell_-|}\|y_{\ell_-}\|=\mathcal{O}( \e^{-aN_*}) \|y_{\ell_-}\|,\\
\|T_4\|\omega_n^{-1} & \le  C|t_{\ell}|\e^{a|n-\ell_-|-\alpha|n-\hat{\ell}|} \le
C\e^{-(\alpha-a)|\hat{\ell}-\hat{\ell}_+|}\|y_{\ell_-}\|=
\mathcal{O}( \e^{-aN_*}) \|y_{\ell_-}\|, \\
\|T_5\|\omega_n^{-1} & \le C \|\sum_{m\in I(s)}\bar{x}_{\ell_--m}\| \|y_{\ell_-}\|
= \mathcal{O}( \e^{-\alpha N_*}) \|y_{\ell_-}\|.
\end{aligned}
\end{equation*} 
Since the operator $\mathcal{L}_0$ has a Green's function with an exponentially
decaying kernel we infer  from Lemma \ref{plateau}
\[ \|d_{\Z}\|_{\omega}\le C \e^{-a N_*} \|y_{\ell_-}\|.
\]
Combining this with \eqref{ymrest} gives
\begin{equation*}
\begin{aligned}
\|\sum_{m\in I(s)}y_{n-m}^m - y_n\| \; &\le  \|d_n\| + \sum_{\ell,\hat{\ell}\neq m
\in I(s)} \|y_{n-m}^m\| \\
& \le  C\Big( \omega_n \e^{-a N_*} + \sum_{\ell,\hat{\ell}\neq m
\in I(s)}\e^{-a|\ell_- -m|}\Big) \|y_{\ell_-}\|= \mathcal{O}(\e^{-a N_*})\omega_n \|y_{\ell_-}\|.
\end{aligned}
\end{equation*}
This finally proves the $y$-estimates in \eqref{strongpoint}.
\end{proof} 


\begin{appendix}
\section{Auxiliary results}\label{ap1}
\begin{lemma}(Banach Lemma) \label{banachlemma}
Let $X,Y$ be Banach spaces and let $A\in L[X,Y]$, $B_-,B_+\in L[Y,X]$
be bounded linear operators such that
\[ \| I_{Y}-AB_+\| < 1, \quad \|I_{X}-B_-A\| <1.
\]
Then $A$ is a homeomorphism with
\begin{equation} \label{estinvA}
\|A^{-1}\| \le \min\left(\frac{\|B_+\|}{1-\|I_Y-AB_+\|},
\frac{\|B_-\|}{1-\|I_X-B_-A\|}\right).
\end{equation}
\end{lemma}

\begin{proof} Note that $y=(I_Y-AB_+)y + r$ has a unique solution $y$
for every $r\in Y$. Then $y$ satisfies
\[ \|y\|\le \frac{\|r\|}{1-\|I_Y-AB_+\|}
\]
and $x=B_+y$ solves $Ax=r$. To prove uniqueness, note that any solution
$x$ of $Ax=r$ solves $x=(I_X-B_-A)x+B_-r$. Since $I_X-B_-A$ is also
contractive the solution is unique and the estimates follow.
\end{proof}

A key tool in  the proofs of Lemma \ref{L1} and Theorem \ref{Th5.1} is the 
following
 quantitative version of the Lipschitz inverse mapping theorem, 
cf.\  \cite{ir02}.
\begin{theorem}\label{Va2}
Assume $Y$ and $Z$ are Banach spaces,
 $F \in C^1(Y,Z)$ and 
$F'(y_0)$ is for $y_0 \in Y$ a homeomorphism. Let 
$\kappa, \ \sigma, \ \delta > 0$ be three constants, such that the 
following estimates hold:
\begin{eqnarray}\label{invlip}
\big\|F'(y) - F'(y_0)\big\| 
&\le& \kappa < \sigma \le \frac 1 {\big\|F'(y_0)^{-1}\big\|}
\quad \forall y \in B_\delta(y_0),\label{V1}\\ \label{Fresidual}
\big\|F(y_0)\big\| &\le& (\sigma - \kappa)\delta \label{V2}. 
\end{eqnarray}   
Then $F$ has a unique zero
$\bar y \in B_\delta(y_0)$ and the following inequalities are satisfied
\begin{eqnarray}
\big\|F'(y)^{-1} \big\| &\le& \frac 1 {\sigma - \kappa} \quad 
\forall y \in B_\delta(y_0) \label{A1},\\
\|y_1 - y_2\| &\le& \frac 1 {\sigma - \kappa} \big\|F(y_1) - F(y_2)\big\|
\quad \forall y_1, \ y_2 \in B_\delta(y_0). \label{A2} 
\end{eqnarray}
\end{theorem}

We collect some well known results on exponential
dichotomies from \cite{pa88}.
Denote by $\Phi$ the solution operator of the linear difference
equation
\begin{equation}\label{app4}
y_{n+1} = A_n y_n,\quad n \in \Z,
\end{equation}
which is defined as 
\begin{equation} \label{defPhi}
\Phi(n,m) := \left\{
\begin{array}{cl}
A_{n-1} \ldots A_m,\quad & \text{for }
n>m,\\
I, &\text{for } n = m,\\
A_n^{-1} \ldots  A_{m-1}^{-1},\quad&
\text{for }  n < m.
\end{array}
\right.
\end{equation}

\begin{definition}\label{b14}
The linear difference equation \eqref{app4} 
with invertible matrices $A_n \in \R^{k,k}$ 
has an \textbf{exponential dichotomy}
with data $(K,\alpha,P_n^{s},P_n^u)$
on an interval $J\subset \Z$, if there exist two families of projectors
$P_n^{s}$ and $P_n^u = I-P_n^{s}$ and constants
$K,\ \alpha >0$, such that the following statements hold:
$$
P_n^s \Phi(n,m) = \Phi(n,m)P_m^s \quad \forall n,m \in J,
$$
$$
\begin{array}{ccl}
\|\Phi(n,m)P_m^{s}\| &\le& K\e^{-\alpha(n-m)}
\vspace{2mm}\\
\|\Phi(m,n)P_n^u\| &\le& K\e^{-\alpha(n-m)}
\end{array}
\quad \forall n \ge m,\ n,m\in J.
$$
\end{definition}

\begin{theorem} (\textrm{Roughness Theorem, cf.\ \cite[Proposition
  2.10]{pa88}}) \label{rough} 
Assume that the difference equation
$$
y_{n+1} = A_n y_n,\quad A_n \in \R^{k,k} \text{ invertible}, \quad
\|A_n^{-1}\| \le M\ \forall n \in J
$$
with an interval $J \subseteq \Z$,
has an exponential dichotomy with data $(K,\alpha,P_n^s,P_n^u)$.
Suppose $B_n \in \R^{k,k}$ satisfies $\|B_n\|
\le b$ for all $n \in J$ with a sufficiently small $b$. 
Then $A_n + B_n$ is invertible and the perturbed difference equation 
$$
y_{n+1} = (A_n + B_n) y_n
$$
has an exponential dichotomy on $J$.
\end{theorem}

\end{appendix}

\nocite{kl97}
\nocite{kl98a}
\nocite{bhkz04}
\nocite{p00}
\nocite{gh90}
\nocite{lm95}

\small
\bibliographystyle{abbrv}
\bibliography{Literatur}

\end{document}